\documentclass[a4paper,11pt]{article}

\usepackage[utf8]{inputenc}
\usepackage[T1]{fontenc}
\usepackage{graphicx}
\usepackage{amsmath,amsfonts,amssymb,amsthm}
\usepackage{mathtools}
\usepackage{mathrsfs}
\usepackage{xcolor}
\usepackage{hyperref}
\usepackage{booktabs}
\usepackage{bbold}
\usepackage{bm}
\usepackage{color}
\usepackage[font=footnotesize]{caption}
\usepackage{enumitem}
\usepackage{empheq}
\usepackage{framed}
\usepackage{fullpage}
\usepackage{hyperref}
\usepackage{soul}

\mathtoolsset{showonlyrefs}

\definecolor{myred}{HTML}{880000}
\definecolor{mygreen}{HTML}{008800}
\definecolor{myblue}{HTML}{000088}
\definecolor{linkblue}{HTML}{0000BB}

\hypersetup{
  colorlinks,
  linkcolor={myred},
  citecolor={myblue},
  urlcolor={linkblue}
}

\newcommand{\R}{\mathbf R}

\newcommand{\E}{{\mathbf E}}
\renewcommand{\P}{\mathbf P}
\newcommand{\Var}{\operatorname{Var}}
\newcommand{\ip}[2]{\ensuremath{\left\langle {#1}, {#2} \right\rangle}}

\newcommand{\norm}[1]{\lVert {#1} \rVert}

\renewcommand{\leq}{\leqslant}
\renewcommand{\geq}{\geqslant}
\renewcommand{\le}{\leqslant}
\renewcommand{\ge}{\geqslant}
\newcommand{\mleq}{\preccurlyeq}

\newcommand{\argmin}{\mathop{\mathrm{arg}\,\mathrm{min}}}

\newcommand{\wt}{\widetilde}
\newcommand{\wh}{\widehat}

\newcommand{\di}{\mathrm{d}}
\DeclareMathOperator{\tr}{Tr}

\newcommand{\var}{\Var}
\newcommand{\cov}{\mathrm{Cov}}
\newcommand{\cond}{|}
\newcommand{\indic}[1]{\bm 1 ( #1 )}

\newcommand{\poissondist}{\mathsf{Poi}}
\newcommand{\binomdist}{\mathsf{Bin}}

\newcommand{\eps}{\varepsilon}
\newcommand{\excessrisk}{\mathcal{E}}
\newcommand{\mom}{\mathrm{MOM}}
\newcommand\MOM{\mom}

\newcommand{\F}{\mathcal{F}}
\newcommand{\Flin}{\mathcal{F}_{\mathrm{lin}}}
\newcommand{\G}{\mathcal{G}}

\renewcommand{\H}{\mathcal{H}}
\newcommand{\freg}{f_{\mathrm{reg}}}
\newcommand{\fw}{\wh g_{\mathrm{FW}}}
\newcommand{\trunc}{m}
\newcommand{\ftrunc}{\wh g_{\trunc}}

\newcommand{\ferm}{\wh g_{\mathrm{erm}}}
\newcommand{\werm}{\wh w_{\mathrm{erm}}}
\renewcommand{\top}{\mathsf{T}}

\usepackage{xspace}
\newcommand{\ie}{\textit{i.e.}\@\xspace}

\newcommand{\iid}{i.i.d.\@\xspace}
\newcommand{\as}{almost surely\@\xspace}
\newcommand{\resp}{respectively{ }}

\newcommand{\opnorm}[1]{\| {#1} \|_{\mathrm{op}}}

\newtheorem{proposition}{Proposition}
\newtheorem{theorem}{Theorem}
\newtheorem*{theorem*}{Theorem}
\newtheorem{lemma}{Lemma}

\newtheorem{informaltheorem}{Theorem}

\theoremstyle{definition}

\newtheorem{assumption}{Assumption}
\theoremstyle{remark}

\title{Distribution-Free Robust Linear Regression
}

\author{
  Jaouad Mourtada\thanks{CREST, ENSAE, Institut Polytechnique de Paris, France, \href{mailto:jaouad.mourtada@ensae.fr}{jaouad.mourtada@ensae.fr}}
  \qquad
  Tomas Va\v{s}kevi\v{c}ius\thanks{Department of Statistics, University of
  Oxford, United Kingdom,
\href{mailto:tomas.vaskevicius@stats.ox.ac.uk}{tomas.vaskevicius@stats.ox.ac.uk}}
  \qquad
  Nikita Zhivotovskiy\thanks{Department of Mathematics, ETH Z\"{u}rich, Switzerland, \href{mailto:nikita.zhivotovskii@math.ethz.ch}{nikita.zhivotovskii@math.ethz.ch}}
}
\date{\today}

\begin{document}

\maketitle
 
\begin{abstract}
  We study random design linear regression with no assumptions on the distribution of the covariates and with a heavy-tailed response variable.
  In this distribution-free regression setting, we show that boundedness of the conditional second moment of the response given the covariates is a necessary and sufficient condition for achieving nontrivial guarantees.
  As a starting point, we prove an optimal version of the classical in-expectation bound for the truncated least squares estimator due to Gy\"{o}rfi, Kohler, Krzy\.{z}ak, and Walk.
  However, we show that this procedure fails with constant probability for some distributions despite its optimal in-expectation performance.
  Then, combining the ideas of truncated least squares, median-of-means procedures, and aggregation theory, we construct a non-linear estimator achieving excess risk of order $d/n$ with an optimal sub-exponential tail.
  While existing approaches to linear regression for heavy-tailed distributions focus on proper estimators that return linear functions, we highlight that the improperness of our procedure is necessary for attaining nontrivial guarantees in the distribution-free setting.
  
  \vspace{5pt}
  \textbf{MSC2020 Subject Classifications:}
        62J05; 62G35; 68Q32.
        
  \vspace{5pt}
\end{abstract}

\section{Introduction}
\label{sec:introduction}

In the random design regression problem, one has
access to $n$ input-output pairs $(X_{i}, Y_{i}) \in \R^{d} \times \R$
sampled \iid from some unknown distribution $P$.
We call any function $g : \R^{d} \to \R$ a \emph{predictor} and measure its quality
via the expected squared error $R(g) = \E (g(X) - Y)^{2}$, also called
\emph{risk}. Based on the sample $S_{n} = (X_{i}, Y_{i})_{i=1}^{n}$,
we aim to construct a \emph{predictor} $\widehat{g}$ whose risk $R( \wh g)$
is small. Since the risk is relative to the problem difficulty, it is
customary to compare it with the best possible risk achievable via some
reference class of functions;
in this work, we mainly focus on the class of
all linear functions $\Flin = \{\ip{w}{\cdot} : w \in \R^{d} \}$.
We therefore consider the \emph{excess risk} of the estimator $\wh g$, defined by
\begin{equation}
  \label{eq:exces-risk}
  \excessrisk(\widehat{g})
  = R(\widehat{g})
  - \inf_{g \in \Flin} R(g).
\end{equation}
One can assume without loss of generality that the infimum above is attained by some linear function $\ip{w^{*}}{\cdot}$, where $w^{*} \in
\R^{d}$.
Note that, since $\wh g$ depends on the random sample $S_n$, the excess risk~\eqref{eq:exces-risk} is also random.
In this paper, we study non-asymptotic bounds on $\excessrisk(\widehat{g})$, both in expectation and with high-probability, for suitable choices of
estimators $\widehat{g}$.

Arguably the most natural and commonly studied procedure is the linear least squares estimator, which selects a linear function $\wh g \in \Flin$ minimizing the (possibly regularized) \emph{empirical risk} $\wh R (g) = \frac{1}{n}\sum_{i=1}^{n} (g(X_{i}) - Y_{i})^2$.
Estimators based on empirical risk minimization (ERM) are known to achieve 
optimal $d/n$ excess risk in expectation under well-behaved covariates $X$ and assuming that the noise random
variable $\xi = Y - \ip{w^*}{X}$ is not correlated with $X$ (see, for example,
\cite{breiman1983many,caponnetto2007optimal,mourtada2019exact}).
The work of Oliveira \cite{oliveira2016lower} 
highlights that the usual sub-Gaussian assumption on the distribution of $X$ can be significantly relaxed in the context of linear regression.
For example, an $L_4$--$L_2$ norm equivalence of the form
\begin{equation}
  \label{eq:l4l2}
  \big( \E \ip{X}{w}^4 \big)^{1/4}
  \le \kappa \big( \E \ip{X}{w}^2 \big)^{1/2}
  , \quad \textrm{for all}\quad w \in \R^d,
\end{equation}
for some constant $\kappa > 0$ is sufficient to achieve the $d/n$ excess risk
rate under additional assumptions on the noise.
Indeed, \cite{oliveira2016lower} shows that under \eqref{eq:l4l2}, we have a high-probability control over the lower tail of the sample covariance matrix, used in the analysis of linear least squares.
This moment equivalence assumption and its variations have become standard tools in the recent literature on robust linear regression
(see, for example, \cite{lecue2016performance, hsu2016heavy, catoni2016pac,
  klivans18a, lugosi2019risk, cherapanamjeri2020optimal, pensia2020robust}).
However, as several authors have recently pointed out,
the kurtosis constant $\kappa$ satisfying the inequality \eqref{eq:l4l2} may
depend on the dimension $d$, leading to suboptimal bounds
\cite{catoni2016pac,oliveira2016lower, lecue2016performance}.
In particular, Saumard \cite{saumard2018optimality} shows that the slightly weaker small-ball condition fails to hold (with a dimension-free constant) for dictionaries consisting of many classical function bases, such as histograms and wavelets, leading to bounds with a suboptimal dependence on the dimension $d$.
In fact, in some
cases this behavior is inherent to
empirical risk minimization, which has recently been shown by the second and the third authors of this paper
\cite{vavskevivcius2020suboptimality} to incur suboptimal excess risk even in a favorable setup where both $X$ and $Y$ are almost surely bounded.
This naturally brings the question of whether distributional assumptions on $X$ such as the condition~\eqref{eq:l4l2} can be relaxed, and if so, what a corresponding minimal assumption on the distribution of the response variable would be.
It is a priori unclear whether non-trivial guarantees are at all possible without imposing any assumption on $X$.

To better contextualize our aims, we turn to a recent line of work initiated
by Catoni \cite{catoni2012challenging}, concerned with
the design and analysis of statistical estimators robust to heavy-tailed distributions.
The ERM strategy is known to fail in this setting due to its sensitivity to the large fluctuations and atypical samples arising from heavy-tailed distributions.
Thus, different techniques and procedures are required to handle such distributions. 
We call the excess risk $\excessrisk(\widehat{g})$ the accuracy of an estimator
$\widehat{g}$; the confidence of $\widehat{g}$ for an error rate of
$\eps$ is equal to $\P(\excessrisk(\widehat{g}) \leq \eps)$.
Robust statistical learning aims to design
procedures with optimal accuracy/confidence trade-off under
minimal distributional assumptions. In the context of linear regression, the optimal trade-off is usually achieved via the bounds on
$\excessrisk(\widehat{g})$ of order $(d + \log(1/\delta))/n$ that hold with
probability at least $1 - \delta$; in particular, such bounds match the
performance of ERM for sub-Gaussian distributions. Using either PAC-Bayesian
truncations \cite{audibert2011robust,catoni2016pac} or the median-of-means tournaments
\cite{lugosi2019risk}, it has been shown that the optimal accuracy/confidence trade-off can be achieved under the $L_4$--$L_2$ moment equivalence assumption \eqref{eq:l4l2} together with some additional assumptions on the noise variable $\xi = Y - \ip{w^*}{X}$.
We remark that existing procedures for heavy-tailed regression select a function within the class $\Flin$.
However, as we shall shortly explain, any such procedure fails in our distribution-free setting.
We can now formulate the question studied in this paper.

\begin{framed}
  Is it possible to predict as well as the best linear predictor in $\Flin$
  without any assumption on the distribution of
  the covariates $X$, while maintaining the optimal accuracy/confidence
  trade-off? If so, what is the minimal assumption on the response variable $Y$ allowing this?
\end{framed}

Independently of the literature on robustness to
heavy-tails, two existing results provide
non-asymptotic guarantees without assumptions on $X$, albeit only in expectation.
Of course, once all the assumptions on $X$ are dropped, the conditional distribution of $Y$ given $X$, consisting of a probability kernel $(P_{Y | X = x})_{x \in \R^d}$, needs to be restricted.
We now state the only assumption considered in this work;
it is satisfied when $Y$ is bounded, but also allows to consider heavy-tailed distributions.
\begin{assumption}
  \label{assumption:y-given-x}
  The 
  conditional distribution 
  of $Y$ given $X$
  satisfies, for some $m > 0$,
  \[
  \sup_{x \in \R^d} \E [Y^2 \vert X =x ] \leq m^2.
  \]
\end{assumption}

The first result not involving any explicit restrictions on the distribution of $X$ is a classical bound for the truncated linear least squares estimator $\ftrunc$ due to Gy\"{o}rfi, Kohler, Krzy\.{z}ak, and Walk \cite[Theorem 11.3]{gyorfi2002nonparametric}
(we defer the exact definition to Section~\ref{sec:clipped-erm}). Under
Assumption \ref{assumption:y-given-x}, their result states that
 \begin{equation}
    \label{eq:gyorfi-bound}
    \E\,R(\ftrunc) - \inf\limits_{g \in \Flin}R(g) \leq c\,\frac{m^{2} d (\log n + 1)}{n}
    + 7 \Big( \inf_{f\in \Flin} R (f) - R (\freg) \Big).
\end{equation}
Here the expectation is taken with respect to the random sample $S_n$, $c>0$ is an absolute constant and $\freg$ is the regression function given by $\freg (x) = \E [Y \cond X=x]$.
The bound \eqref{eq:gyorfi-bound} is a standard benchmark for
several communities. Applications of this result are known in
mathematical finance \cite{zanger2013quantitative}, optimal control
\cite{belomestny2018advanced} and variance reduction \cite{gobet2016monte,
gobet2017adaptive}; there are known improvements of this result under different assumptions
\cite{cohen2013stability, comte2020regression}.

The second bound does not depend on the distribution of $X$ and is due to Forster and Warmuth \cite{forster2002relative}; this estimator originates in the online learning
literature and is obtained via a modification of the renowned
non-linear Vovk-Azoury-Warmuth forecaster \cite{vovk2001competitive, azoury2001relative}.
The Forster-Warmuth estimator, denoted by $\widehat{g}_{\mathrm{FW}}$,
satisfies the following expected excess risk bound
\begin{equation}
  \label{eq:fwexcesrisk}
  \E R(\widehat{g}_{\mathrm{FW}}) - \inf\limits_{g \in \Flin}R(g)
  \leq \frac{2 \| Y \|_{L_\infty}^2 d}{n}.
\end{equation}
Of course, the assumption $\|Y\|_{L_{\infty}} \leq m$ is stronger than
Assumption~\ref{assumption:y-given-x}.
However, an
inspection
of the proof in \cite{forster2002relative} shows that
Assumption~\ref{assumption:y-given-x} suffices to obtain the above
in-expectation performance of this algorithm with $\| Y \|_{L_\infty}^2$ replaced by $m^2$.

We are now ready to present informal statements of our main findings.
In our first result, we prove that the term
$7 \left( \inf_{f\in \Flin} R (f) - R (\freg) \right)$
as well as the excess $\log n$ factor appearing in the bound
\eqref{eq:gyorfi-bound} for the truncated linear
least squares estimator can be removed.
\begin{informaltheorem}[Informal]
  \label{thm:informal-clipping-thm}
  Suppose that Assumption~\ref{assumption:y-given-x} holds and let
  $\ftrunc$ denote
  the truncated least squares estimator of Gy\"{o}rfi, Kohler, Krzy\.{z}ak, and Walk.
  Then, we have
  $$
    \E R(\ftrunc) - \inf\limits_{g \in \Flin}R(g) \le \frac{8m^2 d}{n}.
  $$
  Moreover, Assumption~\ref{assumption:y-given-x} ensures the same guarantee for the Forster-Warmuth estimator $\fw$.
\end{informaltheorem}

One may notice that even though the bound of Theorem \ref{thm:informal-clipping-thm} scales as $d/n$, the usual dependence on the variance of the noise variable as is in, for example, \cite{breiman1983many} is replaced by the dependence on $m^2$. It can be shown (see Proposition \ref{prop:statistical-lower-bound}) that if only Assumption~\ref{assumption:y-given-x} holds, then the dependence on $m^2$ is unavoidable in general even if the problem is noise-free so that the variance of the noise is equal to zero.
Moreover, if we only impose Assumption~\ref{assumption:y-given-x},
then any statistical estimator that selects predictors from $\Flin$
(such an estimator is called \emph{proper}) is bound to fail. This fact can be established using the recent result of
Shamir \cite[Theorem 3]{shamir2015sample}, and it remains
true even when $d=1$ and the response variable $Y$ is bounded almost surely.
This observation separates our setup from the existing literature where only
proper estimators are studied for convex classes such as $\Flin$ even
in the heavy-tailed scenarios (see, for example, \cite{catoni2016pac,
lugosi2019risk, mendelson2017aggregation, mendelson2019unrestricted}).

The bounds of Theorem~\ref{thm:informal-clipping-thm} guarantee
that the \emph{expected} excess risk is at most of order $d/n$ under
Assumption~\ref{assumption:y-given-x}.
However,
it is also desirable to obtain
\emph{high-probability} upper bounds on the excess risk, with logarithmic dependence on the confidence level $\delta$. It is not unreasonable to expect that
the in-expectation guarantees of either $\ftrunc$ or $\widehat{g}_{\mathrm{FW}}$
transfer to analogous high-probability bounds, at least whenever $Y$ is bounded almost
surely. Our second result shows that this is, unfortunately, not the case.
Indeed, both algorithms fail to achieve high probability upper bounds in a strong sense:
they both fail with constant probability.
This does not contradict the previous in-expectation bounds, since neither $\ftrunc$ nor $\widehat{g}_{\mathrm{FW}}$ belong to the linear class
$\Flin$, so the random variable $R(\widehat{g}) - \inf_{g \in \Flin}R(g)$ can take negative values.
Consequently, Markov's inequality cannot be applied to obtain deviation
bounds of $m^2 d/ (\delta n)$ with probability $1-\delta$.

\begin{informaltheorem}[Informal]
  \label{thm:informal-high-probability-failure}
  Let $\widehat{g}$ denote either $\ftrunc$ or $\widehat{g}_{\mathrm{FW}}$.
  There exist universal constants $p \in (0, 1), c > 0$ such that the following holds.
  For any $d \geq 1$ and $m > 0$, there exists a distribution $P$ satisfying
  $\norm{Y}_{L_{\infty} (P)} \leq m$ such that, with probability at least $p$,
  $$
    R(\widehat{g}) - \inf\limits_{g \in \Flin}R(g) \ge c\,m^2.
  $$
\end{informaltheorem}

Theorem~\ref{thm:informal-high-probability-failure} raises the question of whether achieving high-probability guarantees in our distribution-free setting is at all possible.
Indeed, all known high-probability guarantees on linear aggregation
problems impose some restrictions on $X$.
We show that there is, in fact, a procedure that achieves an optimal excess risk guarantee (up to a logarithmic factor) with a sub-exponential tail.
The following theorem is the main positive result of this paper.

\begin{informaltheorem}[Informal]
  \label{thm:informal-high-probability-estimator}
  Suppose that Assumption~\ref{assumption:y-given-x} holds.
  There exists an absolute constant $c > 0$ such that the following holds. For any
  confidence level $\delta \in (0,1)$, there
  exists an improper estimator $\widehat{g}$ (depending on $\delta$ {and $m$}) such that
  $$
   \P\left(
     R(\widehat{g}) - \inf\limits_{g \in \Flin}R(g) \le c\,\frac{m^2(d\log(n/d) + \log(1/\delta))}{n}
   \right) \geq 1 - \delta.
  $$
\end{informaltheorem}
Theorem~\ref{thm:informal-high-probability-estimator} demonstrates that robust
learning of linear classes is possible with no restriction on the distribution of $X$, and under weak tail assumptions on the conditional distribution of the
response variable $Y$ given covariates $X$.
Moreover, we show in Section~\ref{sec:statistical-lower-bound} that Assumption~\ref{assumption:y-given-x} is necessary to obtain any non-trivial guarantee without assumptions on $X$. The estimator of Theorem~\ref{thm:informal-high-probability-estimator} naturally leverages the ideas of
the analysis of truncated linear functions \cite[Chapter
11]{gyorfi2002nonparametric}, skeleton estimators \cite[Section 28.3]{devroye2013probabilistic}, \cite{rakhlin2017empirical},
the deviation optimal model selection aggregation procedures \cite{audibert2008deviation, lecue2009aggregation, mendelson2017aggregation}, min-max estimators \cite{audibert2011robust, lecue2020robust},
and the median-of-means tournaments \cite{lugosi2019risk}.
An extended discussion is deferred to Section~\ref{sec:optimal-estimator}.

\subsection{Summary of contributions and structure of the paper}
\begin{itemize}
    \item In Section \ref{sec:proper-vs-improper}, we discuss known results on distribution-free learning of linear classes. 
    
    \item In Section \ref{sec:clipped-erm}, we show that the classical bound of Gy\"{o}rfi, Kohler, Krzy\.{z}ak, and Walk \cite[Theorem 11.3]{gyorfi2002nonparametric} for the truncated linear least squares estimator can be improved to achieve the optimal $m^2d/n$ bound in expectation.
    
    \item In Section \ref{sec:constant-probability-failure}, we establish that the truncated least squares and Forster-Warmuth estimators are both deviation-suboptimal. In particular, we construct a distribution with almost surely bounded response variable $Y$, under which both estimators incur an excess risk of order $m^2$ with constant probability.
    
    \item Section \ref{sec:optimal-estimator} is split into three parts. In Section~\ref{sec:warm-up:-known}, we consider a simplified setting with a known covariance structure. Combining Tsybakov's projection estimator \cite{tsybakov2003optimal} with the robust mean estimator of Lugosi and Mendelson \cite{lugosi2019mean}, we provide an estimator attaining the optimal rate $d/n$ with the optimal dependence on the confidence parameter. In Section~\ref{sec:general-case}, we drop the simplifying assumption of known covariance structure and present our main positive result -- a distribution-free deviation-optimal estimator robust to heavy-tailed responses. In Section~\ref{sec:extensions}, we discuss possible extensions of this result.
      In particular, we show that an adaptation of our linear regression procedure yields an estimator with deviation-optimal rates for heavy-tailed model selection aggregation under Assumption~\ref{assumption:y-given-x}.
    
    \item Section \ref{sec:statistical-lower-bound} is devoted to establishing the necessity of Assumption~\ref{assumption:y-given-x}.
     We show, in particular, that if $\E[Y^2|X]$ is unbounded, no estimator can achieve non-trivial excess risk guarantees. In addition, we establish that the dependence on $m^{2}$ in our upper bounds is unavoidable.
    
    \item Section \ref{section:deferred-proof} contains deferred proofs of lemmas appearing in the previous sections.
\end{itemize}

\subsection{Related work}
\label{sec:related-work}

\paragraph{Analysis of least squares estimators.}
The most standard 
approach to regression problems is the 
least squares principle, where one selects the predictor achieving the best fit to data within some predefined class of functions.
A large body of work is devoted to analyzing and obtaining guarantees on its performance,
in its most classical form, relying on the fact that the empirical risk 
provides a good approximation of its population counterpart. 
This is typically established when the underlying distribution is sufficiently well-behaved (for instance, bounded or light-tailed), using tools from empirical process theory.
For this point of view to statistical learning, we refer to the standard textbooks \cite{geer2000empirical,massart2007concentration,koltchinskii2011oracle,wainwright2019high}.
It should be noted that statistical analysis of linear regression has also been treated via a complementary approach of stochastic approximation; see, for instance, the works \cite{walk1989convergence, gyorfi1996averaged, dieuleveut2016nonparametric} and references therein.

A recent line of research has established that empirical minimization can perform well under significantly weaker assumptions. 
Our starting point is the work of Oliveira \cite{oliveira2016lower}, where in the context of linear regression the usual sub-Gaussian assumption on $X$ is replaced by a significantly weaker $L_4$--$L_2$ moment equivalence assumption \eqref{eq:l4l2}.
In particular, such an assumption does not even require the existence of any moments of $X$
higher than the fourth. Variations of this assumption have became the standard
tool in the recent literature on linear regression
\cite{lecue2016performance, hsu2016heavy, catoni2016pac,
lugosi2019risk, klivans18a, mourtada2019exact, cherapanamjeri2020optimal,
pensia2020robust}. The seminal work of Mendelson \cite{mendelson2015smallball} introduced a more general
condition, called the \emph{small-ball} assumption. In most of the aforementioned papers, the analysis is
performed for empirical risk minimization, which usually does not lead to the optimal
accuracy/confidence trade-off. The papers \cite{audibert2010linear, audibert2011robust} provide
the optimal confidence for ERM, albeit under stronger moment equivalence assumptions than
that of \eqref{eq:l4l2}.
The $L_{4}$--$L_{2}$ moment equivalence is also important in the robust
covariance estimation problem \cite{catoni2016pac, mendelson2020robust, ostrovskii2019affine}.

It has been recently observed that the absolute constants involved in the
moment equivalence and the small-ball assumptions can behave badly in some
cases. First, Saumard \cite{saumard2018optimality} shows that the small-ball
condition is unsuitable for some important classes leading to suboptimal
performance of ERM. Further, the work \cite{catoni2016pac} (see also the
discussion in \cite{oliveira2016lower} and \cite{lecue2016performance})
discusses that the kurtosis constant $\kappa$ in the moment assumptions similar
to \eqref{eq:l4l2} can depend on the dimension and affect the bounds
negatively. The recent paper \cite{vavskevivcius2020suboptimality} shows this
suboptimal behavior in the context of linear regression, even in a favorable
setup where both $X$ and $Y$ are bounded. There is a growing interest in
further relaxing these assumptions and refining the underlying methods
\cite{saumard2018optimality, catoni2017dimension, mendelson2020extending,
mendelson2019unrestricted, chinot2019robust, mourtada2019exact,
mendelson2020bounded}. In particular, the works \cite{catoni2017dimension,
mendelson2020bounded} replace moment equivalence assumptions by the bounds
on the $L_p$ moments for $p \ge 4$. This is closer to the setting we are aiming
for in this paper.

\paragraph{Robustness to heavy-tailed distributions.}

In a broad sense, robustness encompasses the study and design of statistical
estimation procedures exhibiting certain stability properties under the
existence of ``outlier'' points in the observed sample. For a classical
perspective on robustness, originating from the work of Tukey
\cite{tukey1960survey} and building on the ideas of
contaminated models, influence functions and breakdown points, we refer to the
standard books \cite{hampel1980robust, huber1981robust,
rousseeuw1987robust}.

In contrast to the classical perspective, our work falls within the recent body
of work initiated by Catoni \cite{catoni2012challenging}, where the term
robustness is to be understood specifically as robustness to heavy-tailed distributions
(rather than, for example, adversarial contamination of the sample).
The starting point of this direction is the question of mean estimation, where
informally, one aims to construct statistical estimators performing as well
as the sample mean does for Gaussian samples, all while making as weak
distributional assumptions as possible. Several ways of constructing such
estimators (called sub-Gaussian estimators) have been proposed in the literature.
The most widespread approach is based on the median-of-means estimators,
which appear first independently in \cite{nemirovsky1983problem, jerrum1986random, alon1999space}
and were further developed in the works of
\cite{minsker2015geometric, devroye2016sub, lugosi2019near, lugosi2019sub}.
Other techniques include the Catoni's estimator and its extensions \cite{catoni2012challenging, catoni2017dimension} or the trimmed means \cite{lugosi2021robust}.
We refer to the survey \cite{lugosi2019mean} for further details and
references. For a complementary survey focusing on the computational aspects
see \cite{diakonikolas2019recent}.

The central ideas behind the robust mean estimation found their
applications in many related problems such as regression \cite{hsu2016heavy,
brownlees2015empirical, chinot2019robust, lugosi2019risk, lecue2020robust,
chinot2019robust, minsker2019excess, mendelson2019unrestricted}, covariance
estimation \cite{catoni2016pac, mendelson2020robust, ostrovskii2019affine} and clustering
\cite{brownlees2015empirical, klochkov2020robust}. In the context of linear
regression, the first works showing the optimal accuracy/confidence
trade-off under weak assumptions are attributed to Audibert and Catoni \cite{audibert2010linear,
audibert2011robust} and were further extended in \cite{catoni2016pac,
catoni2017dimension}; these papers are based on PAC-Bayesian truncations.

\paragraph{Distribution-free linear regression.}
Distribution-free non-asymptotic excess risk bounds take their roots in the
PAC-learning framework \cite{Vapnik74, valiant1984theory}, where historically
the binary loss is studied the most. Because of its boundedness, excess risk
bounds in such setups can be obtained without any assumptions on the
distribution of $(X, Y)$. In the context of non-parametric regression with the
squared loss, only asymptotic consistency results are possible under truly
minimal assumptions on the underlying distribution
(see the book \cite{gyorfi2002nonparametric}). In fact, the standard notions of
universal consistency \cite[Section 1.6 and Chapter
10]{gyorfi2002nonparametric} involve only the assumption $\E Y^2 < \infty$ and
no assumptions on the distribution of $X$. The distribution-free nature of this
notion is one of our motivations.  A notable non-asymptotic result in this
direction is \cite[Theorem 11.3]{gyorfi2002nonparametric}, where an inexact oracle inequality \eqref{eq:gyorfi-bound} is proved without
any explicit assumptions on the distribution of $X$.

Another direction originates from the online learning literature (see
\cite{cesa2006prediction} for background on this topic). For instance, when
both $X$ and $Y$ are bounded, the renowned Vovk-Azoury-Warmuth forecaster
\cite{vovk2001competitive, azoury2001relative} can be used to provide
excess risk bounds of order $d/n$ in our setup even when the aforementioned
moment equivalence constants behave badly with respect to the dimension.
This observation has been recently explored in \cite{vavskevivcius2020suboptimality}.
For linear regression, the Forster-Warmuth algorithm \cite{forster2002relative},
which is in turn a modification of the Vovk-Azoury-Warmuth forecaster,
leads to the only known exact oracle inequality without imposing any assumptions on $X$.
\subsection{Notation}
\label{sec:setting-notations}
We now set the notation. 
We let $P = P_{(X,Y)}$ be the joint distribution on $\R^d \times \R$ (with $d \geq 1$) of a random pair $(X,Y)$.
The joint distribution $P$ itself can be decomposed into two components, namely the marginal distribution $P_X$ of $X$
(a distribution on $\R^d$),
as well as the conditional distribution of $Y$ given $X$, consisting of a (measurable) probability kernel $(P_{Y | X = x})_{x \in \R^d}$, where for $x \in \R^d$, $P_{Y \cond X=x}$ is a distribution on $\R$.

For a real random variable $Z$ and $p \geq 1$, we denote $\| Z \|_{L_p} = \E [ |Z|^p ]^{1/p}$, while for a measurable function $f : \R^d \to \R$, we set $\| f \|_{L_p} = \| f \|_{L_p (P_X)} = \| f (X) \|_{L_p}$.

The risk of a measurable function $f : \R^d \to \R$ is by definition $R (f) = \E ( f(X) - Y )^2 = \| f(X) - Y \|_{L_2}^2$.
It is known that the risk is minimized by the regression function $\freg$ given by $\freg (x) = \E [Y \cond X=x] = \int_\R y P_{Y\cond X=x} (\di y)$. 

Absolute constants are denoted by $c, c_1, \ldots$ and may change from line to line.
For a real square matrix $A$, let $\tr(A)$ denote its trace, $\opnorm{A}$ its operator norm, $A^\top$ its transpose and $A^{\dagger}$ its Moore–Penrose inverse.
In what follows, $\langle \cdot, \cdot\rangle$ denotes the canonical inner product in $\R^d$ and $\norm{\cdot}$ stands for the Euclidean norm. 
For any two functions (or random variables) $f, g$ the symbol $f \lesssim g$ (or $g \gtrsim f$) means that there is an absolute constant $c$ such that $f \le cg$ on the entire domain.
For a pair of symmetric matrices $A, B$, the symbol $A \mleq B$ means that $B - A$ is positive semi-definite.

We consider the class $\Flin = \{\ip{w}{\cdot} : w \in \R^{d} \}$ of linear functions. 
Throughout, our assumptions will imply that $R(0) = \E Y^2$ is finite (regardless of $P_X$); hence, so is the minimal risk in $\Flin$, namely $\inf_{f \in \Flin} R (f)$ is finite.
In this case, for $f \in \Flin$ given by $f(x) = \langle w, x\rangle$ its risk $R (f)$ is finite if and only if
$\| \langle w, X\rangle \|_{L_2} < + \infty$, and the set of such $w \in \R^d$ is a subspace of $\R^d$, which coincides with $\R^d$ itself if and only if $\E \| X \|^2 < + \infty$.
When the latter condition holds, one can define the covariance of $X$ as $\cov(X) = \E (X - \E X)(X - \E X)^\top$ and the Gram matrix of $X$ as $\Sigma = \E X X^\top$;
the minimizers $f$ of the risk in $\Flin$ are then the functions $\langle w, \cdot\rangle$, where $w$ are solutions of the equation $\Sigma w = \E [Y X]$. The last quantity is well-defined since $\E |Y| \| X \| \leq \| Y \|_{L_2} \E [ \| X\|^2 ]^{1/2}$. 

Given the observed sample $S_{n} = (X_{i}, Y_{i})_{i=1}^{n}$,
the aim is to construct a predictor (usually called an estimator) $\widehat{g}$ whose risk $R(\widehat{g})$ is small.
A learning procedure is a measurable function mapping a sample in $(\R^d \times \R)^n$ to a measurable function $\R^d \to \R$.
In what follows, we avoid measurability issues and use a standard convention that all events appearing in the probabilistic statements are measurable.
Given a sample $S_{n} = (X_{i}, Y_{i})_{i=1}^{n}$, we usually write $\wh g$ for the function $\wh g (S_n)$. Finally, we remark that since the sample $S_n$ is random, the function $\wh g = \wh g(S_n)$ is also random and so is $R(\wh g)$.

\section{Distribution-free linear regression: 
known results
}
\label{sec:proper-vs-improper}
In this section, we set the context for the rest of this work, by reviewing relevant existing results on distribution-free linear prediction, and framing them in our setting (through minor modifications).
We remark that the bounds we are about to discuss hold in expectation, whereas we will also be concerned with high-probability guarantees.
As will be seen in Section~\ref{sec:constant-probability-failure}, the distinction between the two is not innocuous, as existing procedures achieving distribution-free expected excess risk bounds do not possess
matching guarantees in deviation.

\paragraph{Limitations of proper estimators.}
Recall that in the context of our work, a learning procedure is called \emph{proper}
if it always returns an element of the class $\Flin$ (that is, a linear function); otherwise, it is called
\emph{improper} or
\emph{non-linear}.
The importance of
considering improper estimators stems from a fundamental limitation of proper procedures in our distribution-free setting.
Specifically, it follows
from the work of Shamir \cite[Theorem
  3]{shamir2015sample} that for any proper estimator
$\widehat{g}_{\mathrm{proper}}$, there exists a distribution of $(X, Y)$ with the response
variable $Y$ almost surely bounded by $m$, for which
\begin{equation}
  \label{eq:lowerbound-proper-shamir}
  \E R (\wh g_{\mathrm{proper}}) - \inf_{g \in \Flin} R (g)
  \gtrsim m^2.
\end{equation}
Thus,
even when the response is bounded,
no proper learning procedure can improve (up to universal constants) over the risk trivially achieved by the zero function, without some restrictions on the distribution of covariates.
As discussed in the introduction, this negative result already rules out many procedures introduced and analyzed in the statistical learning and robust estimation literature, including empirical risk minimization and refinements thereof.

\paragraph{Learning
  with known covariance structure.}

We now discuss a simplified setting, in which guarantees can be obtained quite directly.
Specifically, assume that the \emph{covariance structure} of the distribution $P_X$, namely, the map $w \mapsto \E \langle w, X \rangle^{2}$ (which can take infinite values), is known.
As noted in Section~\ref{sec:setting-notations}, we can restrict our attention to the linear subspace where
the above map takes finite values. Thus, we may assume without loss of generality
that the covariance matrix $\Sigma = \E XX^{\mathsf{T}}$ exists.
In addition, up to restricting to the orthogonal of the nullspace
$\{ w \in \R^d : \Sigma w = 0 \}$, we may assume in what follows that the
covariance matrix $\Sigma$ is invertible. Hence, the unique minimizer of the
risk $R(f)$ in $\Flin$ is given by $f^* = \langle w^*, \cdot \rangle$, 
where  $w^* = \Sigma^{-1} \E [YX]$. In addition, the excess risk of any linear function
$f = \langle w, \cdot \rangle$
is given by the following
identity:
\begin{equation}
\label{eq:ex-risk-ident}
R (f) - \inf \limits_{g \in \Flin}R(g) = \| \Sigma^{1/2} (w - w^*) \|^2.
\end{equation}

The key simplification provided by the knowledge of $\Sigma$ is that random-design
linear regression reduces to multivariate mean estimation.
To see this, consider the change of variables
$\theta = \Sigma^{1/2}w$ and notice that
the excess risk~\eqref{eq:ex-risk-ident} is then equal to $\| \theta - \theta^* \|^2$,
where $\theta^{*} = \E U$ for $U = Y\Sigma^{-1/2}X$.
Using
$\Sigma$, an \iid sample $(X_{i}, Y_{i})_{i=1}^{n}$ can
be turned into an \iid sample $(U_{i})_{i=1}^{n}$, with $U_{i} = Y_i \Sigma^{-1/2} X_i$ distributed
as $U$. One can thus estimate $\E U$ by the sample mean
$\frac{1}{n}\sum_{i=1}^{n}U_{i}$. This leads to the \emph{projection estimator} for
our original problem, defined as
\begin{equation}
  \label{eq:projection-estimator}
  \widehat{g}_{\mathrm{proj}}(x) = \langle \widehat{w}, x \rangle\quad\text{where}\quad
  \widehat{w} = \Sigma^{-1} \cdot \frac{1}{n}\sum_{i=1}^{n}Y_{i}X_{i}.
\end{equation}
Under Assumption~\ref{assumption:y-given-x}, we have
\[
  \E UU^\top
  = \E \big[ \E [Y^2 \cond X] \Sigma^{-1/2} X X^\top \Sigma^{-1/2} \big]
  \mleq m^2 I_d ,
\]
and in particular $\tr(\cov (U)) \leq m^2 d$ and $\opnorm{\cov (U)} \leq m^2$.
Applying the first inequality to the empirical mean estimator of $\theta^* =
\E U$ leads to the following guarantee for the projection estimator,
which corresponds up to minor changes in assumptions\footnote{
  Specifically, \cite{tsybakov2003optimal} assumes that the noise
  $Y - \freg (X)$ is independent of $X$, but the same proof applies
  when replacing this assumption by the conditional moment bound of
  Assumption~\ref{assumption:y-given-x}.}
to the result of Tsybakov~\cite[Theorem~4]{tsybakov2003optimal}:
\begin{equation}
  \label{eq:known-covariance-expectation}
  \E R (\widehat{g}_{\mathrm{proj}}) - \inf_{g \in \Flin} R (g)
  \leq \frac{m^2 d}{n}.
\end{equation}
It is worth noting that there is no contradiction between the lower bound
\eqref{eq:lowerbound-proper-shamir} and the above upper bound.
Indeed, the projection estimator, while proper, relies on the a priori knowledge of $\Sigma$, which is unavailable in the typical statistical learning setting.
This implies in particular that the 
knowledge of $\Sigma$ is sufficient to
avoid the previous failure of proper procedures.
In
this work, the simplified setting with known covariance
serves as a benchmark that we aim to match in the general
case, where nothing is known a priori about the distribution of $X$.

\paragraph{Upper bounds in expectation via non-linear predictors.}
As mentioned in the introduction, and with the exception of the aforementioned known covariance setting, there
are
two known results stating non-trivial in-expectation
guarantees without restrictions on the distribution of $X$.
These guarantees are achieved, respectively, by the truncated linear least squares estimator $\ftrunc$ and the Forster-Warmuth estimator $\fw$, which we now define formally.

First, consider the linear least squares
estimator $\ferm = \argmin_{g \in \Flin} \wh R (g) = \langle \werm, \cdot \rangle$, where
\begin{equation}
  \label{eq:erm-least-squares}
  \werm
  = \Big( \sum_{i=1}^n X_i X_i^\top \Big)^\dagger \Big( \sum_{i=1}^n Y_i X_i
  \Big)
  = \wh \Sigma_n^\dagger \cdot \frac{1}{n} \sum_{i=1}^n Y_i X_i
  \, ,
\end{equation}
with $\wh \Sigma_n = \frac{1}{n} \sum_{i=1}^n X_i X_i^\top$.
Given a threshold $m > 0$, the truncated least squares estimator $\ftrunc$
returns the prediction of the linear function
$\langle \widehat{w}_{\mathrm{erm}}, \cdot \rangle$,
truncated
to $[-m, m]$. That is,
\begin{equation}
  \label{eq:def-truncated-least-squares}
  \ftrunc (x)
  = \max ( -m, \min ( m, \langle \werm, x \rangle ) ).
\end{equation}
We now turn to
the Forster-Warmuth estimator.
Given the sample $S_{n} = (X_{i}, Y_{i})_{i=1}^{n}$, define the
\emph{leverage score} of a point $x \in \R^{d}$ by
$h_{n}(x) = \langle (n\wh {\Sigma}_{n} + xx^\top)^\dagger x, x\rangle$.
The Forster-Warmuth estimator is then defined by reweighing predictions of
$\langle \werm, \cdot \rangle$ by a function of the
statistical leverage of the input point $x$:
\begin{equation}
  \label{eq:def-forster-warmuth}
  \fw (x)
  = \big( 1 - h_n (x) \big)^2 \cdot \langle
  \widehat{w}_{\mathrm{erm}}, x\rangle .
\end{equation}

Recall the guarantees on the risk of these procedures, stated in the introduction.
Specifically, 
under Assumption~\ref{assumption:y-given-x},
the truncated least squares estimator satisfies the 
oracle inequality~\eqref{eq:gyorfi-bound}, while the Forster-Warmuth estimator achieves the excess risk bound~\eqref{eq:fwexcesrisk}.
Note that both procedures are improper, as they introduce non-linearities in the prediction function, either through truncation or through the leverage correction.

As discussed in the recent works
\cite{mourtada2019exact, vavskevivcius2020suboptimality}, the risk of the least squares procedure is large when leverage scores are uneven and correlate with the noise.
While this configuration is ruled out under distributional assumptions such as moment equivalences, it can actually occur even under boundedness constraints, leading to poor performance~\cite{vavskevivcius2020suboptimality}.
Both non-linearities partially mitigate the shortcomings of the least squares estimator, by adjusting its predictions at high-leverage points, which are the most unstable and lead to large errors.
These corrections allow these procedures to achieve
in-expectation bounds, even for
unfavorable distributions on which ordinary least squares fail.

\section{An improved bound for
  truncated least squares}
\label{sec:clipped-erm}

As discussed at the end of Section~\ref{sec:proper-vs-improper}, the non-linearities introduced by the truncated least squares and Forster-Warmuth estimators 
aim to mitigate
the instability of ERM predictions at high-leverage points. 
The more sophisticated Forster-Warmuth procedure (which relies on an explicit leverage correction), however, leads to a better excess risk guarantee. 
Indeed, the risk guarantee of $\ftrunc$ takes the form of an inexact oracle inequality, suffering from the approximation error
term $ \inf_{f\in \Flin} R (f) - R (\freg)$.
This type of guarantee only ensures that the procedure approaches the performance of the best linear function in the nearly well-specified case, where the true regression function is almost linear.
While reasonable
in low-dimensional nonparametric estimation~\cite{gyorfi2002nonparametric} (with appropriate linear spaces), such an assumption is generally restrictive in high-dimensional problems and is not satisfied in our setting.
Unfortunately, the proof technique employed in \cite{gyorfi2002nonparametric}
can only yield inexact oracle inequalities, and hence,
no straightforward modification to their argument can
match guarantees of $\widehat{g}_{\mathrm{FW}}$ given by \eqref{eq:fwexcesrisk}.

A natural question remains of whether the gap between the existing in-expectation performance guarantees
given by~\eqref{eq:gyorfi-bound} and \eqref{eq:fwexcesrisk} is intrinsic to the estimators $\ftrunc$ and $\widehat{g}_{\mathrm{FW}}$, or whether it is a byproduct of suboptimal analysis of the performance of the simpler procedure $\ftrunc$.
In the theorem below, we show that truncated least squares estimator
indeed matches the statistical performance of the Forster-Warmuth algorithm.
Our proof is based on
a leave-one-out argument
akin to the one used to
prove the upper bound \eqref{eq:fwexcesrisk} in \cite[Section 3]{forster2002relative}.
We remark that leave-one-out arguments have a long history; see the references \cite[Chapter 6]{Vapnik74} and \cite{haussler1994predicting}.
\begin{theorem}
  \label{thm:clipped-erm-loo-bound}
  Suppose that Assumption~\ref{assumption:y-given-x} holds and let
  $\ftrunc$ denote the truncated least squares estimator~\eqref{eq:def-truncated-least-squares}. Then, we have
  \begin{equation}
    \label{eq:bound-truncated}
    \E R (\ftrunc) - \inf_{f \in \Flin} R (f)
    \leq \frac{8 m^{2}d}{n+1}.
  \end{equation}
\end{theorem}

\begin{proof}
  To simplify the presentation, we introduce
  additional notation.
  Let
  $S_{n+1} = (X_i, Y_i)_{i=1}^{n+1}$
  denote an \iid sample of size $n+1$.
  For any $j \in \{ 1, \dots, n + 1 \}$, let $S_{n+1}^{(j)} = (X_i,Y_i)_{i=1, i\neq j}^{n+1}$ be the dataset obtained by removing the $j$-th sample.
  On the sample $S_{n+1}$ (\resp $S_{n+1}^{(j)}$), we 
  define
  the minimal norm
  empirical risk minimizer $\wt g$ (\resp $\wt g^{(j)}$) and its truncated variant $\wt g_\trunc$ (\resp $\wt g_\trunc^{(j)}$).

  Since $S_{n+1}$ is an \iid sample, for every $j \in \{ 1, \dots, n+1\}$, $S_{n+1}^{(j)}$ has the same distribution as $S_n = S_{n+1}^{(n+1)}$ (so that $\wt g_\trunc^{(j)}$ has the same distribution as $\ftrunc = \wt g_\trunc^{(n+1)}$), and is independent of $Z_j = (X_j, Y_j)$.
  This implies that the expected excess risk of $\ftrunc$ can be bounded
  as follows:
  \begin{align}
    \E \, \excessrisk (\ftrunc)
    &=
      \E_{S_{n+1}}
        \big( \wt g_\trunc^{(n+1)} (X_{n+1}) - Y_{n+1}\big)^{2}
        - \inf_{g \in \Flin}
      \E_{Z_{n+1}}
      \left(g(X_{n+1}) - Y_{n+1}\right)^{2}\\
    &=
      \E_{S_{n+1}}\bigg[
        \frac{1}{n+1}\sum_{j=1}^{n+1}
        \left(\wt {g}^{(j)}_{\trunc}(X_{j}) - Y_{j}\right)^{2}
      \bigg]
        - \inf_{g \in \Flin}
        \E_{S_{n+1}} \bigg[
          \frac{1}{n+1}\sum_{j=1}^{n+1}
          \left(g(X_{j}) - Y_{j}\right)^{2}
        \bigg] \\
    &\leq
      \E_{S_{n+1}}\bigg[
        \frac{1}{n+1}\sum_{j=1}^{n+1}
        \left(\wt g^{(j)}_{\trunc}(X_{j}) - Y_{j}\right)^{2}
        - \left(\wt g (X_{j}) - Y_{j}\right)^{2}
       \bigg], \label{eq:loo-intermediate-step}
  \end{align}
  where the last line follows from the definition of $\wt g$.
  Now, define the leverage $h_j$ of the point $X_j$ among $X_1, \dots, X_{n+1}$ by
  \begin{equation}
    \label{eq:def-leverage-proof}
    h_j
    = \bigg\langle \bigg( \sum_{i=1}^{n+1} X_i X_i^\top \bigg)^\dagger X_j, X_j \bigg\rangle
    \in [0, 1]. 
  \end{equation}
  An explicit 
  computation---postponed to the end of the proof---shows that for every $j$,  
  \begin{equation}
    \label{eq:proof-identity-leverage}
    \wt g (X_j)
    = (1 - h_j) \, \wt g^{(j)} (X_j) + h_j Y_j.
  \end{equation}
  Plugging~\eqref{eq:proof-identity-leverage} into the bound~\eqref{eq:loo-intermediate-step}, we obtain
  \begin{equation}
    \label{eq:loo-after-sherman-morrison}
    \E \, \excessrisk (\ftrunc)
    \leq
      \E \bigg[
        \frac{1}{n+1}\sum_{j=1}^{n+1}
        \left(\wt {g}^{(j)}_{\trunc}(X_{j}) - Y_{j}\right)^{2}
        - (1-h_{j})^{2}\left( \wt {g}^{(j)} (X_{j}) - Y_{j}\right)^{2}
       \bigg].
  \end{equation}
  By Assumption~\ref{assumption:y-given-x} and Jensen's inequality we have
  $\sup_{x \in \R^{d}} |\freg (x)| \leq m$.
  It follows that $(\wt g_m^{(j)} (X_j) - \freg (X_j))^2 \leq (\wt g^{(j)} (X_j) - \freg (X_j))^2$, so that 
  \begin{align}
    \label{eq:proof-truncation-reduce}
    &\E \Big[ (1 - h_j)^2 \big( \wt {g}^{(j)} (X_{j}) - Y_{j}\big)^{2} \, \big| \, S_{n+1}^{(j)}, X_j \Big] \\
    &= (1 - h_j)^2 \Big( (\wt g^{(j)} (X_j) - \freg (X_j))^2 + \E \big[ ( \freg (X_{j}) - Y_{j})^{2} |  S_{n+1}^{(j)}, X_j \big] \Big) \\
    &\geq \E \Big[ (1 - h_j)^2 \big( \wt {g}^{(j)}_{m} (X_{j}) - Y_{j}\big)^{2} \, \big| \, S_{n+1}^{(j)}, X_j \Big].      
  \end{align}
  Plugging the above in the upper bound~\eqref{eq:loo-after-sherman-morrison}, we proceed as follows
  \begin{align}
    \E \, \excessrisk (\wh g_{m})
    &\leq
      \E \bigg[
        \frac{1}{n+1}\sum_{j=1}^{n+1}
        \big(\wt g^{(j)}_{m}(X_{j}) - Y_{j}\big)^{2}
        - (1-h_{j})^{2} \big(\wt g_{m}^{(j)}(X_{j}) - Y_{j}\big)^{2}
       \bigg] \\
    &\leq
      \E \bigg[
        \frac{1}{n+1}\sum_{j=1}^{n+1}
        2h_{j}\big(\wt {g}^{(j)}_{m}(X_{j}) - Y_{j}\big)^{2}
       \bigg] \\
    &\leq
      8m^{2} \E \bigg[
        \frac{1}{n+1}\sum_{j=1}^{n+1}
        h_{j}
       \bigg] 
           \leq 8\frac{m^{2}d}{n+1},
  \end{align}
  where the penultimate step follows from Jensen's inequality combined with
  Assumption~\ref{assumption:y-given-x} and the last step
  follows from the bound $\sum_{j=1}^{n+1} h_j = \tr \big[ \big( \sum_{i=1}^{n+1} X_i X_i^\top \big)^\dagger \big( \sum_{i=1}^{n+1} X_i X_i^\top \big) \big] \leq d$.

  We now conclude by showing the identity~\eqref{eq:proof-identity-leverage}.
  First, define
  \begin{align*}
    \wt \Sigma
    = \sum_{i=1}^{n+1} X_{i}X_{i}^{\mathsf{T}} ,\quad
    \wt \Sigma^{(j)} = \wt \Sigma - X_{j}X_{j}^\top,\quad
    b = \sum_{i=1}^{n+1} Y_{i} X_{i},
    \enskip\text{and}\enskip
    b^{(j)} = b - Y_{j} X_{j},
  \end{align*}
  so that
  \begin{equation*}
    \wt g (X_j)
    = \langle \wt \Sigma^\dagger b, X_j \rangle, \quad
    \wt g^{(j)} (X_j)
    = \langle \big( \wt \Sigma^{(j)} \big)^\dagger b^{(j)}, X_j \rangle, \quad \text{and} \quad
    h_j = \langle \wt \Sigma^\dagger X_j, X_j \rangle.
  \end{equation*}
  Note that~\eqref{eq:proof-identity-leverage} is an identity, and up to restricting to the linear span of $(X_1, \dots, X_{n+1})$ we may assume that $\wt \Sigma$ is invertible.
  In addition, if $X_j$ does not belong to the linear span of $(X_i)_{i=1,i\neq j}^{n+1}$, namely, if $\wt \Sigma^{(j)}$ is singular, then it can be shown that $h_j = 1$ and $\wt g (X_j) = Y_j$ (since $\wt g$ minimizes the empirical risk on $S_{n+1}$, and $ g (X_j)$ can be set freely without affecting the other predictions), so that~\eqref{eq:proof-identity-leverage} holds.
  Therefore, we may assume that $\wt \Sigma^{(j)}$ is invertible.
  Using the definition and the Sherman-Morrison formula, as $h_j \in [0,1)$, we obtain
  \begin{align}
    \wt g^{(j)} (X_{j})
    &= \bigg\langle
      \bigg( \wt \Sigma^{-1} +
      \frac{\wt \Sigma^{-1} X_{j}X_{j}^\top \wt \Sigma^{-1}}{1- h_{j}} \bigg) \left(b - Y_{j} X_{j}\right), 
      {X_{j}} \bigg\rangle \\
    &= \wt g (X_{j}) +
      \frac{h_{j}}{1-h_{j}} \wt g (X_{j})
      - h_{j}Y_{j} - \frac{h_{j}^{2}}{1-h_{j}}Y_{j} \\
    &= \frac{1}{1-h_{j}} \wt g (X_{j}) -
      \frac{h_{j}}{1-h_{j}}Y_{j};
  \end{align}
  rearranging the last equality yields~\eqref{eq:proof-identity-leverage}, concluding the proof.
\end{proof}

\section{Failure of previous estimators with constant probability
}
\label{sec:constant-probability-failure}

As discussed in Section~\ref{sec:proper-vs-improper},
Assumption~\ref{assumption:y-given-x} suffices to ensure that the Forster-Warmuth
estimator \cite{forster2002relative} achieves an expected excess risk bound of order $m^2d/n$
irrespective of the distribution of $X$. Our results established in
Section~\ref{sec:clipped-erm} demonstrate the same conclusion for the truncated
least squares estimator of \cite[Theorem 11.3]{gyorfi2002nonparametric}.
In addition to the guarantees in expectation, high-probability or tail bounds are desirable,
as they provide a control on the probability of failure of the estimator.
The following theorem shows that in fact, none of the two procedures
satisfy meaningful high-probability guarantees, in a rather strong sense.

\begin{theorem}
  \label{thm:lower-bound-deviation-trunc-fw}
  Fix the dimension $d = 1$.
  There exist absolute constants $c > 0$ and $n_0 \geq 2$
  such that the following holds.
  For any $n \geq n_0$, there is a distribution $P = P(n)$ of $(X,Y)$ with $\| Y \|_{L_\infty} \leq m$, such that if $\wh g$ is either the truncated least squares estimator
  \eqref{eq:def-truncated-least-squares} or the Forster-Warmuth estimator
  \eqref{eq:def-forster-warmuth}, computed on an \iid sample $S_n$, then
  \begin{equation}
    \label{eq:lower-bound-fw-truncated}
    \P \Big( R (\wh g) - \inf_{g \in \Flin} R (g) \geq c \, m^2 \Big)
    \geq c
    \, .
  \end{equation}
\end{theorem}

Note that under Assumption~\ref{assumption:y-given-x}, the trivial, identically
zero function has risk at most $\E Y^{2} \leq m^{2}$.
Theorem~\ref{thm:lower-bound-deviation-trunc-fw} states that, with constant
probability, the truncated least squares and the Forster-Warmuth estimators
incur a constant excess risk of the same order.
At the first sight, this property
may seem incompatible with expected excess risk bounds of order
$d/n$.
However, one should keep in mind that the estimators in question are improper
(returning predictors outside of the class $\Flin$), so that the excess risk
may well take negative values; the expected excess risk remains small
due to the fact that positive and negative values essentially compensate in
expectation, regardless of the distribution.

A related phenomenon was observed in the context of model
selection-type aggregation by Audibert~\cite{audibert2008deviation}, who showed
that the (improper) progressive mixture rule~\cite{yang2000mixing,catoni2004statistical}, known to achieve fast rates in expectation, exhibits slow rates in deviation.
In our context the failure in deviation is even more severe, as the
excess risk is of constant order, rather than exhibiting slow rates.

\begin{proof} 
For any $n \geq n_0$, let $P = P(n)$ be the distribution of $(X,Y)$ satisfying
  \begin{equation}
    \label{eq:bad-distribution}
    (X,Y)
    = \left\{
      \begin{array}{ll}
        (1,m) & \mbox{ with probability } 1-\frac{1}{n}\,; \\
        (\sqrt{n}, 0) & \mbox{ with probability } \frac{1}{n}\,.
      \end{array}
    \right.
  \end{equation}
  By homogeneity, we may assume that $m=1$.
  For any $w \in \R$, set $g_w(x) = w\cdot x$. We have
  \begin{equation*}
    R (g_w)
    = \Big( 1 - \frac{1}{n} \Big) (w-1)^2 + \frac{1}{n} (w \sqrt{n})^2
    = \Big( 1 - \frac{1}{n} \Big) (w-1)^2 + w^2
    \, .
  \end{equation*}
  It follows that the risk of the best linear predictor is equal to
  \begin{equation}
    \label{eq:proof-failure-bestlinear}
    \inf_{w \in \R} R (g_w)
    = \frac{1 - 1/n}{2 - 1/n}
    \leq \frac{1}{2}
    \, .
  \end{equation}
  In addition, let $K = K_n$ denote the number of indices $i=1, \dots, n$ such that $X_i = \sqrt{n}$.
  The empirical risk writes
  \begin{equation*}
    \wh R_n (g_w)
    = \Big( 1 - \frac{K}{n} \Big) (w - 1)^2 + K w^2,
    \quad \mbox{and so} \quad
    \werm
    = \argmin_{w \in \R} \wh R_n (g_w)
    = \frac{1 - K/n}{K + 1 - K/n}
    \, .
  \end{equation*}
  In particular, $0 \leq \werm \leq 1/(K+1)$.
  Now, note that if $\wh g$ denotes either the truncated least squares
  \eqref{eq:def-truncated-least-squares} or the
  Forster-Warmuth estimator \eqref{eq:def-forster-warmuth},
  then $\wh g(1) \leq \werm \cdot 1 \leq 1/(K+1) \leq 1$, and thus, denoting
  the sample $(X_{i}, Y_{i})_{i=1}^{n}$ by $S_n$, we have
  \begin{equation}
    \label{eq:lowerbound-truncated-fw}
    R (\wh g)
    \geq \E \big[ (\wh g (X) - Y)^2 \indic{X=1} \, \cond \, S_n \big]
    \geq \Big( 1 - \frac{1}{n} \Big) \cdot \Big( \frac{K}{K+1} \Big)^2
    \, .
  \end{equation}
  Thus, under the event $E_n = \{ K_n \geq 4 \}$,
  it follows from~\eqref{eq:proof-failure-bestlinear} and~\eqref{eq:lowerbound-truncated-fw} that for $n \geq 16$,
  \begin{equation}
    \label{eq:lowerbound-excessrisk-fw}
    R (\wh g) - \inf_{g \in \Flin} R (g)
    \geq \Big( 1 - \frac{1}{n} \Big) \cdot \Big( \frac{K}{K+1} \Big)^2 - \frac{1}{2}
    = \Big( 1 - \frac{1}{16} \Big) \cdot \frac{16}{25} - \frac{1}{2}
    = \frac{1}{10}
    \, .
  \end{equation}
  Finally, since $K_n$ follows the binomial distribution
  $\binomdist (n, 1/n)$, the probability $\P (E_n)$ is
  positive for $n \geq 16 \geq 4$. Further, since $K_n$ converges in distribution to
  the Poisson distribution $\poissondist (1)$ as $n \to \infty$, $\P (E_n) \to \P
  (\wt K \geq 4) > 0$ with $\wt K \sim \poissondist (1)$, so that setting $p_0 =
  \inf_{n \geq 16} \P (E_n)$, we have $p_0 > 0$.  This
  concludes the proof
  with $c = \min (p_0, 1/10)$ and $n_0 = 16$.
\end{proof}

\section{An
  optimal robust estimator in the high-probability regime}
\label{sec:optimal-estimator}

In this section we present our main positive result. We show that there is an
estimator achieving an optimal accuracy and sub-exponential tails for the linear
class $\Flin$ under Assumption \ref{assumption:y-given-x}. We first consider a
simplified setup where the covariance structure of $X$ is known.

\subsection{Warm-up: known covariance structure}
\label{sec:warm-up:-known}

Following the discussion on the learning model with known covariance structure
in Section~\ref{sec:proper-vs-improper},
we assume in this section that $\Sigma = \E XX^\top$ exists, is invertible and
also known.
Recall the definition of Tsybakov's projection estimator
$\widehat{g}_{\mathrm{proj}}$
\eqref{eq:projection-estimator}. Since this estimator always returns a linear
predictor, its excess risk is non-negative and we may apply
Markov's inequality to show that for any $\delta \in (0, 1)$, it holds that
\[
  \P\left( R (\widehat{g}_{\mathrm{proj}}) - \inf_{g \in \Flin} R (g)
  \leq \frac{m^2 d}{n} \cdot \frac{1}{\delta} \right) \ge 1 - \delta.
\]
An argument similar to the one used in \cite[Proposition 6.2]{catoni2012challenging} can be used to show that this bound is essentially the best we
can hope for the projection estimator, even when $| Y | \leq m$ \as.

Fortunately, there is a way to modify this estimator and obtain a guarantee with
sub-exponential tails.
The result of Lugosi and Mendelson \cite[Theorem
1]{lugosi2019sub} shows for any $\delta \in (0, 1)$, there exists an estimator $\wh \mu_\delta: (\R^d)^n \to \R$ such that, for any sequence $U_1, \dots, U_n$ of \iid random vectors in $\R^d$ with mean $\mu$ and covariance matrix $\wt \Sigma = \cov (U)$, $\wh \mu_\delta = \wh \mu_\delta (U_1, \dots, U_n)$ satisfies
\begin{equation}
  \label{eq:robust-mean-deviation-bound}
  \P \bigg(\|\widehat{\mu}_{\delta} - \mu\|^2 \le c \, \frac{\tr(\widetilde{\Sigma}) +\opnorm{\widetilde{\Sigma}} \log (1/\delta)}{n}\bigg) \ge 1 - \delta,
\end{equation}
where $c > 0$ is an absolute constant.
Now, introduce the \emph{robust} projection estimator
\begin{equation}
  \label{eq:robust-projection}
  \widetilde w
  = \Sigma^{-1/2} \cdot
  \widehat{\mu}_\delta \left( Y_1 \Sigma^{-1/2} X_1, \dots, Y_n \Sigma^{-1/2} X_n \right) ,
\end{equation}
and consider the following result.

\begin{proposition}
  \label{prop:robust-subgaussian}
  There is an absolute constant $c > 0$ such that the following is true. Suppose that Assumption~\ref{assumption:y-given-x} holds. Then, the robust
  projection estimator
  $\wh g = \langle\wt w, \cdot\rangle$ (which is a proper estimator)
  defined in
  \eqref{eq:robust-projection} satisfies
  \begin{equation}
    \label{eq:robust-projection-bound}
    \P \bigg(
    R (\wh g) - \inf_{g \in \Flin} R (g)
    \leq c \, \frac{m^2 \big( d + \log (1/\delta) \big)}{n}
    \bigg)
    \geq 1 - \delta.
  \end{equation}
\end{proposition}
\begin{proof}
  We have shown in Section~\ref{sec:proper-vs-improper}, that under
  Assumption~\ref{assumption:y-given-x}, $\cov (Y\Sigma^{-1/2} X) \mleq m^2
  I_d$. Combining the deviation bound \eqref{eq:robust-mean-deviation-bound}
  and the definition \eqref{eq:robust-projection}
  with the identity \eqref{eq:ex-risk-ident}, we finish the proof.
\end{proof}

The above result serves as a benchmark result for the performance that we aim
to establish in the more realistic setting where the covariance matrix $\Sigma$ is unknown. This is achieved in the next section.

\subsection{Deviation-optimal robust estimator 
}
\label{sec:general-case}

The theorem below is the main positive result of our paper. It demonstrates that
Assumption~\ref{assumption:y-given-x} is a sufficient condition for the
existence of linear regression estimators satisfying an excess risk deviation
inequality with logarithmic dependence on the confidence parameter.
In Section~\ref{sec:statistical-lower-bound}, we
show that Assumption~\ref{assumption:y-given-x} is also necessary.

\begin{theorem}
  \label{thm:aggregalgo} There is an absolute constant $c > 0$ such that the following holds.
  Assume that $n \ge d$. Suppose that Assumption~\ref{assumption:y-given-x} holds and fix any
  $\delta \in (0,1)$. Then, there exists an estimator $\widehat{g}$ depending
  on $\delta$ and $m$ such that the following holds:
  \begin{equation*} \P\left(R(\wh g) -
    \inf\limits_{g \in \Flin}R(g) \le c\, \frac{m^2 (d\log(n/d) + \log
    (1/\delta))}{n}\right) \ge 1 - \delta.
  \end{equation*}
  Moreover, the above bound also holds if the class $\Flin$ is replaced by an arbitrary VC-subgraph class $\mathcal{F}$ of dimension $d$.
\end{theorem}
Before presenting our estimator, we briefly comment on the above theorem.
First, in contrast to existing work on robust linear regression, our estimator
$\wh g$ is improper, even though the underlying linear class is convex.
Second, unlike our previous results presented in this paper,
the bound of Theorem~\ref{thm:aggregalgo} is not specific to the linear class.
In particular, our proof extends without changes
to the family of VC-subgraph classes
(see \cite[Definition 3.6.8]{gine2016mathematical}).
Some recent results in the robust statistics literature apply to
more general classes of functions, including
non-parametric classes (see, for example, \cite{lugosi2019risk,
mendelson2019unrestricted, chinot2019robust}). However, as discussed in
Section~\ref{sec:related-work}, such results are only known to be valid under
the existence of assumptions on $P_{X}$.  Extending our results for more general
classes presents some
challenges; we discuss them in
more detail in Section~\ref{sec:extensions}.
Finally, we note that our estimator depends on the value of $m$.
This assumption simplifies the analysis and is standard in similar contexts (see, for example, \cite[Theorem 11.3]{gyorfi2002nonparametric} and
\cite{mendelson2020bounded}).

We now introduce some additional notation needed to define our estimator.
For any $\varepsilon > 0$ and any class of real-valued functions $\G$ let
$\G_{\varepsilon}$ denote the smallest $\varepsilon$-net of $\G$ with respect
to the empirical $L_1$ distance $\frac{1}{n}\sum_{i=1}^{n}|f(X_{i}) -
g(X_{i})|$.
We only consider $\varepsilon$-nets that are subsets of $\G$.
For the standard definition of an $\varepsilon$-net we refer to \cite[Section 4.2]{Vershynin2016HDP}.
Assume that we have a sample $S = (X_i, Y_i)_{i = 1}^{3n}$
of size $3n$ and denote $S_1 = (X_i, Y_i)_{i = 1}^{n}$, $S_2  = (X_i, Y_i)_{i = n + 1}^{2n}$
and $S_3 = (X_i, Y_i)_{i = 2n + 1}^{3n}$.
Fix any $1 \leq k \leq n$, and assume without loss of generality
that $n / k$ is integer.  Split the set $\{1, \ldots, n\}$ into $k$ blocks $I_1,
\ldots, I_{k}$ of equal size such that $I_j = \{1 + (j - 1)(n/k), \ldots, j
(n/k)\}$.  Fix any function $\ell: \R^{d} \times \R \to \R$, any sample
$S'$ of size $n$, and denote the $i$-th element of $S'$ by $Z_{i} = (X_{i}, Y_{i})$.
The median-of-means estimator
(see also
\cite[Section 2.1]{lugosi2019mean}, \cite{nemirovsky1983problem})
is defined as follows:
\[
  \MOM_{S'}^{k}(\ell) =
  \mathrm{Median} \bigg(\frac{k}{n} \sum\limits_{i \in I_1} \ell(Z_i), \ldots,
  \frac{k}{n} \sum\limits_{i \in I_k} \ell(Z_i)\bigg).
\]
Finally, for any predictor $f : \R^{d} \to \R$, denote the associated loss
function by $\ell_f(Z_{i}) = (f(X_{i}) - Y_{i})^{2}$. We are now ready to
present our estimator.

\begin{framed}
\textbf{The estimator of Theorem \ref{thm:aggregalgo}}
\begin{enumerate}
  \item Split the sample $S$ of size $3n$ into three equal parts $S_1, S_2$
    and $S_3$ as defined above. Use the value $m$ to construct the truncated class
    \begin{equation}
      \label{eq:truncatedclass}
      \overline{\F} = \Big\{ f_{m} : f \in \Flin \Big\},
    \end{equation}
    where recall that $f_{m}$ denotes the truncation of a function $f$ (see
    \eqref{eq:def-truncated-least-squares}).
\item Fix $\varepsilon = \frac{md}{n}$.
  Using the first sample $S_1$, construct an $\varepsilon$-net of $\overline{\F}$
  with respect to the empirical $L_1$ distance and denote it by $\overline{\F}_{\varepsilon}$.
\item Let $c_{1}, c_{2} > 0$ be some specifically chosen absolute constants.
  Fix the number of blocks $k = \lceil c_{1} d(\log(n/d) + \log(1/\delta))\rceil$
  and set $\alpha = c_{2}\sqrt{\frac{m^2(d\log(n/d) + \log (1/\delta))}{n}}$. If $k > n $, then set $\wh g = 0$. Otherwise,
  using the second sample $S_2$ define a random subset of
  $\overline{\F}_{\varepsilon}$ as follows:
  \begin{align*}
    \wh \F &= \Bigl\{f \in \overline{\F}_{\varepsilon}: \forall g
      \in \overline{\F}_{\varepsilon},\
      \MOM_{S_2}^{k}\left(\ell_{f} - \ell_{g}\right)
      \le \alpha\sqrt{\frac{1}{n}\sum_{X_{i} \in S_{2}} (f(X_i) -
    g(X_i))^2} + \alpha^{2} \Bigr\}.
  \end{align*}
\item Define the set $\wh \F_{+}$ consisting of all the mid-points of $\wh \F$,
  that is, $\wh \F_{+}  = (\wh \F + \wh \F)/2 $. Using the third sample $S_3$,
  define our estimator $\wh g$ as
  \[
    \wh g = \argmin_{g \in \wh\F_{+} }\max\limits_{f \in \wh\F_{+}
    }\MOM_{S_{3}}^{k}\left(\ell_{g} - \ell_{f} \right).
  \]
\item Return $\wh g$.
\end{enumerate}
\end{framed}

Our estimator involves a combination of several seemingly disconnected ideas in the
literature.  The truncation step is inspired by the analysis in \cite[Chapter
11]{gyorfi2002nonparametric}, with the difference that we use the
truncation as a preliminary step, rather than as a post-processing of the ERM prediction
(see Theorem~\ref{thm:clipped-erm-loo-bound}).
The second step replaces the original class by an empirical $L_1$
$\varepsilon$-net of the truncated class.  In many
situations, such a construction leads to suboptimal results.
However, since we work with a particular parametric class, this step does not affect the
resulting performance.  The use of the $\varepsilon$-net
$\overline{\F}_{\varepsilon}$ is needed for
technical reasons; we explain the technical aspects in detail in Section~\ref{sec:extensions}.
Our third step is inspired by the median-of-means tournaments introduced in
\cite{lugosi2019risk}.  The main difference with the latter work is that our
truncated class is now non-convex, and to obtain the correct rates of convergence,
we need to adapt the arguments used in the model selection aggregation
literature.  This motivates our fourth step that can be seen as an adaptation of the \emph{star algorithm} \cite{audibert2008deviation} and the two-step
aggregation procedure developed in \cite{lecue2009aggregation, mendelson2017aggregation} to our specific
heavy-tailed setting combined with the idea of min-max formulation of robust estimators \cite{audibert2011robust, lecue2020robust}.
We remark that the idea of combining model selection
aggregation techniques with the median-of-means tournaments has also
recently appeared in \cite{mendelson2019unrestricted}, but under different assumptions.
As we mentioned, the key distinction therein is that the suggested learning procedure collapses to
a proper estimator for convex classes of functions, such as $\Flin$
considered in our work; as discussed in Section~\ref{sec:proper-vs-improper},
for such procedures some restrictions on the distribution of covariates
are required to obtain performance bounds.

The rest of this section is devoted to proving
Theorem~\ref{thm:aggregalgo}. First,
the truncation at the level $m$ can only make the
risk smaller whenever Assumption~\ref{assumption:y-given-x} is satisfied. Indeed, this follows from the identity
\[
R(g) = \E(g(X) - \freg(X))^2 + \E(\freg(X) - Y)^2,
\]
and the fact that $\freg$ is absolutely bounded by $m$.
Therefore, we may focus on bounding
\[
  R(\wh g) - \inf_{g \in \overline{\F}}R(g).
\]
We will now state and comment on some technical lemmas that will be used in our
proof. The proofs of the below lemmas are deferred to
Section~\ref{section:deferred-proof}.

Next, we provide a uniform deviation bound on the $L_{1}$
distances between the elements of $\overline{\mathcal{F}}$. 
\begin{lemma}
  \label{lem:lonedist}
  Assume that $n \ge d$. There is a constant $c > 0$ such that simultaneously for all $f, g \in \overline{\F}$, with probability at least $1 - \delta$, it holds that
  \[
    \E|f(X) - g(X)|
    \le \frac{2}{n}\sum\limits_{i = 1}^n |f(X_i) - g(X_i)| +
    c\left(\frac{m d\log(n/d) + m\log(3/\delta)}{n}\right).
  \]
\end{lemma}
To simplify the statements of the lemmas to follow, for any finite class
$\mathcal{G}$ and for any confidence parameter $\delta \in (0,1)$ define:
\begin{equation}
  \label{eq:alpha-definition}
  \alpha(\mathcal{G}, \delta) = 32\sqrt{
    \frac{m^{2}(\log(2|\mathcal{G}|) +
  \log(4/\delta))}{n}},
\end{equation}
where the sample size $n$ and the value $m$ (of
Assumption~\ref{assumption:y-given-x}) will always be clear from the context.
The next technical lemma provides basic concentration properties of the median-of-means estimators, the proof of which follows from
a combination of uniform
Bernstein's inequality and a median-of-means deviation inequality for mean
estimation \cite[Theorem 2]{lugosi2019mean}.

\begin{lemma}
  \label{lem:concent}
  Suppose that Assumption~\ref{assumption:y-given-x} holds and let $S_{n} =
  (X_{i}, Y_{i})_{i=1}^{n}$ denote an \iid sample.
  Let $\mathcal{G}$ be any finite class of functions whose absolute value is
  bounded by $m$. Fix any $\delta \in (0,1)$, let
  $k=\lceil8\log\frac{2 |\G|^{2}}{\delta}\rceil$ and let $\alpha$ denote any
  upper bound on $\alpha(\mathcal{G}, \delta)$ defined in
  \eqref{eq:alpha-definition}. Then, with probability at
  least $1-\delta$, the following inequalities hold simultaneously
  for any $f,g\in \mathcal{G}$:
  \begin{align}
    \label{eq:secondeq}
    \left|R(f) - R(g) - \MOM^{k}_{S_{n}}\left(\ell_{f}- \ell_{g}\right)\right|
    &\le \alpha\sqrt{\E(f(X) - g(X))^2},
    \\
    \label{eq:firsteq} 
    \left|R(f) - R(g) - \MOM^{k}_{S_{n}}\left(\ell_{f} - \ell_{g}\right)\right|
    &\le  \sqrt{2}\alpha\sqrt{\frac{1}{n}\sum\limits_{i = 1}^n(f(X_i) - g(X_i))^2}
    + \alpha^2,
    \\
    \frac{1}{n}\sum_{i=1}^{n}(f(X_{i}) - g(X_{i}))^{2}
    &\leq 2\E(f(X) - g(X))^{2} + \alpha^{2}.
  \end{align}
\end{lemma}

For any class $\mathcal{G}$, define its $L_{2}$ diameter by:
\[
  \mathcal{D}(\G) = \sup_{f, g \in \G}\sqrt{\E(f(X) - g(X))^2}.
\]
As a corollary of the above lemma, we are able to derive some basic properties
of the random set $\wh \F$. In particular, we show that with high probability
the set $\wh \F$ contains the population risk minimizer over the
$\varepsilon$-net $\overline{\F}_{\varepsilon}$. At the same time, we establish
a uniform Bernstein-type bound on the excess risk of the elements of $\wh
\F$, with the role of the variance term played by $\mathcal{D}(\wh \F)$.

\begin{lemma}
   \label{lem:excessrisk}
  Suppose that Assumption~\ref{assumption:y-given-x} holds and let $S_{n} =
  (X_{i}, Y_{i})_{i=1}^{n}$ denote an \iid sample.
  Let $\mathcal{G}$ be any finite class of functions whose absolute value is
  bounded by $m$.
  Fix any $\delta\in(0,1)$, $k=\lceil8\log\frac{2 |\G|^{2}}{\delta}\rceil$ and
  let $\alpha$ denote any upper bound on $\alpha(\mathcal{G}, \delta)$ defined in
  \eqref{eq:alpha-definition}.
  Define the random subset of
  $\mathcal{G}$ :
  \begin{align}
    \wh \G
    = \Bigg\{f \in \G: \textrm{for every}\ g \in \G,\ \MOM_{S_{n}}^{k}\left(\ell_{f} -
      \ell_{g}\right)
      \le \sqrt{2}\alpha\sqrt{\frac{1}{n}\sum_{i=1}^{n}(f(X_{i}) - g(X_{i}))^{2}}
    + \alpha^{2} \Bigg\},
  \end{align}
  Then, the following two
  conditions hold simultaneously, with probability at least $1-\delta$:
  \begin{enumerate}
    \item The function $g^{*} = \argmin_{g \in \G} R(g)$ belongs to the class
      $\wh \G$.
    \item For any $f,g \in \wh \G$, we have
      $R(f) - R(g^*) \leq 4\alpha\mathcal{D}(\wh
      \G) + 5\alpha^{2}$.
  \end{enumerate}
\end{lemma}

Finally, we prove an excess risk bound for the min-max estimator in terms of
the $L_{2}$ diameter of the set over which the estimator is computed.
The intuitive implications of the below lemma are the following. First,
if $\mathcal{D}(\wh \F)$ is of order $1/\sqrt{n}$, the bellow lemma immediately yields
the fast rate of convergence for our estimator $\wh g$.
If, on the other hand, the diameter of $\mathcal{D}$ is much larger than $1/\sqrt{n}$,
then we can exploit the curvature of the quadratic loss and the gain in the approximation error
(due to considering the larger class $\wh \F_{+}$ instead of $\wh \F$)
to prove the desired rate of convergence.

\begin{lemma}
  \label{lemma:minmax-mom-estimator}
  Suppose that Assumption~\ref{assumption:y-given-x} holds and let $S_{n} =
  (X_{i}, Y_{i})_{i=1}^{n}$ denote an \iid sample.
  Let $\mathcal{G}$ be any finite class of functions whose absolute value is
  bounded by $m$. Fix any $\delta \in (0,1)$, let
  $k=\lceil8\log\frac{2 |\G|^{2}}{\delta}\rceil$ and
  let $\alpha$ denote any upper bound on $\alpha(\mathcal{G}, \delta)$ defined in
  \eqref{eq:alpha-definition}.
  Let $\widehat{g}$ be any estimator satisfying
  $$
    \widehat{g}
    \in \argmin_{g \in \G} \max_{f \in \G}
    \MOM_{S_{n}}^{k}(\ell_{g} - \ell_{f}).
  $$
  Let $g^{*} \in \argmin_{g \in \G} R(g)$. Then, with probability at least
  $1-\delta$, it holds that
  \begin{align}
    R(\widehat{g}) \leq R(g^{*}) + 2\alpha\mathcal{D}(\G).
  \end{align}
\end{lemma}

We are now ready to prove Theorem~\ref{thm:aggregalgo}.

\begin{proof}[Proof of Theorem~\ref{thm:aggregalgo}]
  Our proof is split into two parts. First, we approximate the truncated linear
  class $\overline{\F}$ with a finite class, namely, an empirical $L_{1}$
  $\varepsilon$-net constructed using the first third of the dataset denoted by
  $S_{1}$. Then, conditionally on $S_{1}$, we show that our estimator $\wh g$
  achieves the optimal rate of model selection aggregation over the finite class
  $\overline{\F}_{\varepsilon}$, in spite of the lack of assumptions on the covariates
  and the presence of heavy-tailed labels.
  Finally, we note that if the number of median-of-means blocks $k$ is equal to $0$ (\ie, $n \lesssim d(\log(n/d) + \log(1/\delta))$), then we
  may the $0$ function which satisfies the desired bound for such
  sample sizes. Thus, in what follows we assume that $n \gtrsim d(\log(n/d) +
  \log(1/\delta))$.

  \paragraph{The approximation step.}
  Recall that $\overline{\F}_{\varepsilon}$ is an empirical $L_1$
  $\varepsilon$-net of the truncated linear class $\overline{\F}$ constructed
  using the sample $S_{1}$.
  Let $f^* = \argmin_{f \in \overline{\F}} R(f)$ and let $f^{*}_{\varepsilon}$
  be any element of $\overline{\F}_{\varepsilon}$ minimizing the empirical $L_{1}$
  distance to $f^{*}$, that is, we have
  \begin{equation}
    \label{eq:eps-net-approximation-l1}
    \frac{1}{n}\sum_{X_i \in S_{1}}|f^{*}(X_i) - f^{*}_{\varepsilon}(X_i)| \leq
    \varepsilon.
  \end{equation}
  Let $E_{1}$ denote the event of Lemma~\ref{lem:lonedist} applied with respect
  to the sample $S_{1}$ (that contains $n$ points) with the choice of the confidence
  parameter set to $\delta/3$ (thus, $\P(E_{1}) \geq 1 - \delta/3$).
  It follows that on the event $E_{1}$ we have
  \begin{align*}
     &R(f^{*}_{\varepsilon}) - R(f^{*}) \\
    \quad&= 2\E Y(f^*(X) -  f^*_{\eps}(X)) + \E(f^*_{\eps}(X)^ 2 - f^*(X)^2) &
    \\
    \quad&\le 2\E(\E[Y|X](f^*(X) -  f^*_{\eps}(X))) + 2m\E|f^*_{\eps}(X) - f^*(X)|
      &(\text{since } |f^*_{\eps}(X) + f^*(X)| \le 2m)
    \\
    \quad&\le2\E(\sqrt{\E[Y^2|X]}|f^*(X) -  f^*_{\eps}(X)|) + 2m\E|f^*_{\eps}(X) - f^*(X)|
         &(\text{by Jensen's inequality})
    \\
    \quad&\le 4m\E|f^*_{\eps}(X) - f^*(X)|\quad  &(\text{by Assumption~\ref{assumption:y-given-x}})
    \\
    &\le 8m\varepsilon
    + 4mc_{1}\left(\frac{m d \log(n/d) + m\log(9/\delta)}{n}\right)
    &(\text{by \eqref{eq:eps-net-approximation-l1} and Lemma~\ref{lem:lonedist}})
    \\
    &\leq
    12c_{1}\left(\frac{m^{2} d \log(n/d) + m^{2}\log(9/\delta)}{n}\right),
    &(\text{by the definition of }\varepsilon)
  \end{align*}
  where $c_{1}$ is an absolute constant. Observe that on the event $E_{1}$,
  any estimator $\wh g$ satisfies
  \begin{align*}
    R(\wh g) - R(f^{*})
    &\leq R(\wh g) - \min_{f \in \overline{\F}_{\varepsilon}}R(f)
    + R(f^{*}_{\varepsilon}) - R(f^{*}) \\
    &\leq R(\wh g) - \min_{f \in \overline{\F}_{\varepsilon}}R(f)
    + 12c_{1}\left(\frac{m^{2} d \log(n/d) + m^{2}\log(9/\delta)}{n}\right).
  \end{align*}
  From this point onward, we work on the event $E_{1}$. It thus remains to
  prove that with probability $1-2\delta/3$, the estimator $\widehat{g}$
  computed using the remaining $2n$ points split into samples $S_{2}$ and $S_{3}$ satisfies
  \begin{equation}
    \label{eq:aggregation-goal-bound}
    R(\wh g) - \min\limits_{f \in \overline{\F}_{\varepsilon}} R(f)
    \lesssim  \frac{m^{2} d \log(n /d) + m^{2}\log(1/\delta)}{n}.
  \end{equation}
  Since $\overline{\F}_{\varepsilon}$ is a finite class of functions, we now turn
  to the aggregation part of this proof.

  \paragraph{The aggregation step.}
  By the $L_{2}$ covering number bound stated in
  \cite[Theorem 9.4, Theorem 9.5]{gyorfi2002nonparametric},
  which also holds for the empirical $L_1$ distances,
  we have (see the proof of Lemma~\ref{lem:lonedist})
  \[
    \log |\overline{\F}_{\varepsilon}|
    \lesssim d\log \frac{me}{\eps} \lesssim
    d\log (n/d).
  \]
  Note that $|\wh \F_{+}|$ and $|\wh \F|$ are simultaneously upper bounded by
  $|\overline{\F}_{\varepsilon}|^{2}$. For an arbitrary finite class $\G$, recall the
  definition of $\alpha(\mathcal{G}, \delta)$ stated in
  \eqref{eq:alpha-definition}. It follows that there exists some absolute
  constant $c_{2} > 0$ such that $\overline{\alpha}$ defined
  below satisfies
  \begin{equation}
    \label{eq:overline-alpha-dfn}
    \max\left(\alpha(\wh \F, \delta/3), \alpha(\wh \F_{+}, \delta/3)\right)
    \leq \overline{\alpha}
    = c_{2} \sqrt{\frac{m^2 d \log(n/d) + m^2\log(1/\delta)}{n}}.
  \end{equation}
  Thus, $\overline{\alpha}$ defined above will be used in the applications of
  Lemmas~\ref{lem:concent}, \ref{lem:excessrisk} and
  \ref{lemma:minmax-mom-estimator} to follow.

  Let $E_{2}$ be the event of
  Lemma~\ref{lem:excessrisk} applied for the set $\wh \F$ with confidence
  parameter $\delta/3$. In particular, on the event $E_{2}$ we have
  \begin{equation}
    \label{eq:event-E2}
    \argmin_{f \in \overline{\F}_{\varepsilon}} R (f) \in \wh \F,
    \quad\text{and for any }f \in \wh \F \text{ it holds that }
    R(f) \leq \min_{f \in \overline{\F}_{\varepsilon}} R(f) +
    4\overline{\alpha}\mathcal{D}(\wh \F) + 5\overline{\alpha}^{2}.
  \end{equation}

  Conditionally on the sample $S_{2}$, let the set $\wh \F$ defined in the third
  step of our algorithm be fixed. Denote $g^* = \argmin_{g \in \wh \F_{+}}R(f)$,
  where recall that $\wh \F_{+} = (\wh \F +
  \wh \F)/2$. Observe that the $L_{2}$ diameters of $\wh \F$ and $\wh \F_{+}$
  are equal, that is $\mathcal{D}(\wh \F_{+}) = \mathcal{D}(\wh \F)$.
  Let $E_{3}$ be the event of Lemma~\ref{lemma:minmax-mom-estimator} applied to
  the third part of our sample $S_{3}$ and the finite class $\wh \F_{+}$ with
  the confidence parameter set to $\delta/3$.
  Thus, on $E_{3}$ our estimator $\widehat{g}$ satisfies:
  \begin{equation}
  \label{eq:riskofhat}
    R(\wh{g}) \le  R(g^*)  + 2\overline{\alpha}\mathcal{D}(\wh \F).
  \end{equation}
  Now choose any $g, h \in \wh \F$ such that
  $\sqrt{\E(g(X) - h(X))^2} \ge \mathcal{D}(\wh \F)/2$ (such a choice always
  exists by definition of the diameter).
  Since $(g + h)/2 \in \wh \F_{+}$, the parallelogram identity yields
  \begin{align}
    R(g^*)
    &\le R((g + h)/2)
    \\
    &= \frac{1}{2}R(g) + \frac{1}{2}R(h) - \frac{1}{4}\E(g(X) - h(X))^2
    \\
    &\le \frac{1}{2}R(g) + \frac{1}{2}R(h) - \frac{1}{16}\mathcal{D}(\wh \F)^{2}.
    \label{eq:parallelogram}
  \end{align}
  On the event $E_{2}$, applying \eqref{eq:event-E2} for the functions $g$ and
  $h$ we obtain
  \begin{align*}
    \frac{1}{2}R(g) + \frac{1}{2}R(h)
    \leq \min_{f \in \overline{\F}_{\varepsilon}} R(f)
    + 4\overline{\alpha}\mathcal{D}(\wh \F) + 5\overline{\alpha}^{2}.
  \end{align*}
  Combining the above with equations \eqref{eq:riskofhat} and
  \eqref{eq:parallelogram} we have
  \begin{align}
    R(\wh g)
    - \min_{f \in \overline{\F}_{\varepsilon}} R(f)
    \leq 6\overline{\alpha}\mathcal{D}(\wh \F) + 5\overline{\alpha}^{2}
    - \frac{1}{16}\mathcal{D}(\wh \F)^{2}
    \leq 149\overline{\alpha}^{2},
  \end{align}
  where the last step follows by maximizing the quadratic equation with respect
  to $\mathcal{D}(\wh \F)$. Plugging in the definition of
  $\overline{\alpha}$ (see \eqref{eq:overline-alpha-dfn}) we obtain the desired inequality
  \eqref{eq:aggregation-goal-bound}. The proof is complete by taking the union
  bound over the events $E_{1}$, $E_{2}$ and $E_{3}$ defined above.
\end{proof}

\subsection{Some extensions of Theorem~\ref{thm:aggregalgo}}
\label{sec:extensions}

We begin by noting that Theorem~\ref{thm:aggregalgo} holds not only for linear
classes, but more generally, for VC-subgraph classes (without any changes to
our argument presented in the previous section). Indeed, the structure of the
underlying function class only enters our proof though the control on the
empirical covering numbers of its truncated elements; sharp bounds for such
covering numbers are available in \cite[Theorem 9.4]{gyorfi2002nonparametric}.
As a special case, our analysis
covers
finite classes and hence, provides new
results for the problem of model selection aggregation, where a learner is tasked with
constructing a predictor as good as the best one in a given finite class (also
called dictionary) of functions \cite{nemirovski2000topics, tsybakov2003optimal}.
It is arguably the most straightforward problem manifesting statistical separation between
proper and improper learning algorithms (see, for instance,
\cite{catoni2004statistical, juditsky2008learning}). Procedures based on
exponential weighting were shown to attain optimal rates in expectation
\cite{yang2000combining, yang2000mixing,
catoni2004statistical, audibert2009fastrates}, yet they were later shown to be deviation suboptimal
\cite{audibert2008deviation}, in close similarity to our results presented in
Section~\ref{sec:constant-probability-failure}.

We can now formulate the following result, which from the statistical point of
view, generalizes the best known results for the problem of model selection aggregation \cite{audibert2008deviation, lecue2009aggregation}.

\begin{theorem}
  \label{thm:finitedict}
  There is an absolute constant $c > 0$ such that the following holds. Grant Assumption~\ref{assumption:y-given-x}, fix any
  $\delta \in (0,1)$ and let $\F$ be a finite class of possibly
  unbounded functions.
  Then, there exists an estimator $\widehat{g}$ depending
  on $\delta$ and $m$ such that the following holds:
  \begin{equation*}
    \P \left( R(\wh g) -
    \min_{g \in \F} R(g)
    \leq c\, \frac{m^2 (\log|\F| + \log (1/\delta))}{n} \right)
  \ge 1 - \delta.
  \end{equation*}
\end{theorem}

\begin{proof}
  The aggregation algorithm is the same as the estimator of
  Theorem~\ref{thm:aggregalgo} with only two
  differences. First, we 
  skip the step with $\varepsilon$-net
  discretization of the truncated class $\overline{\F}$.
  The second difference is that the number of blocks in median-of-means
  estimators is of order $\log(|\F|/\delta)$ and similarly, the parameter $\alpha$ is
  redefined to be of order $\sqrt{\frac{m^{2}(\log |\F| + \log (1/\delta))}{n}}$.
  The proof follows via the ``aggregation step'' part of the proof of
  Theorem~\ref{thm:aggregalgo}.
\end{proof}

Concerning aggregation with a heavy-tailed response variable, the above result can be compared with the bounds
of Audibert \cite{audibert2009fastrates} and Juditsky, Rigollet and Tsybakov
\cite{juditsky2008learning}.
Assuming that the functions in $\F$ are absolutely
bounded by $1$, and that $\E |Y|^{s} \le m_s$ for some $s \ge 2, m_s > 0$, they
prove an in-expectation bound on $\E R(\widetilde f) - \min_{f \in \F}R(f)$ for
some estimator $\widetilde f$ with the rate of convergence slower than $1/n$.
In contrast, in Theorem \ref{thm:finitedict} we do not assume the boundedness
of $\F$, but require that the conditional second moment of $Y$ is bounded. As a
result, we provide a deviation bound with the  $1/n$ rate of convergence and
logarithmic dependence on the confidence parameter $\delta$. We emphasize again that due to the necessity of improperness for optimal model selection
aggregation, in-expectation results are not easily transferable to deviation
bounds;
the in-expectation guarantees of \cite{juditsky2008learning,audibert2009fastrates} are in fact obtained for variants of the progressive mixture or mirror averaging rule, which is shown by Audibert~\cite{audibert2008deviation} to exhibit suboptimal deviations.
Finally, an argument of Section~\ref{sec:statistical-lower-bound} shows the necessity of Assumption \ref{assumption:y-given-x} in our distribution-free setting for model selection aggregation.

Further extensions of Theorem~\ref{thm:aggregalgo}, particularly, going beyond
VC-subgraph classes present technical challenges. First, obtaining
distribution-free empirical covering number guarantees for truncations of general classes
(as done for $\overline{\F}$ in our case) might be a non-trivial task.
Second, it is well-known (see the discussion in
\cite{rakhlin2017empirical}) that even when only bounded functions are
considered, replacing the original function class by its empirical $\eps$-net
(as done via the function class $\overline{\F}_{\varepsilon}$ in our algorithm)
usually renders the recovery of the correct excess risk rates impossible.
This in turn leads to the final and the most technical problem: if
$\overline{\F}$ is not replaced by $\overline{\F}_{\varepsilon}$, there are no
known ways to obtain an analog of the concentration Lemma~\ref{lem:concent},
while only imposing Assumption~\ref{assumption:y-given-x}.

To expand on the last point, an analog of Lemma~\ref{lem:concent} for general
classes can be approached via the analysis of suprema of localized quadratic and
multiplier processes (see \cite{mendelson2019unrestricted} for related arguments);
specifically, the supremum of the localized process
$\E\sup_{h \in \H_r}\left(\sum_{i = 1}^n\varepsilon_i Y_ih(X_i)\right)$
is difficult to control for general classes under our
assumptions (here $\H_{r}$ denotes localized subsets of the class $\overline{\F} -
\overline{\F}$, see the proof of Lemma~\ref{lem:lonedist} for more details). However, even if the response variable $Y$ is
independent of $X$, the standard in this context application of the multiplier inequality \cite[Lemma
  2.9.1]{van1996weak} introduces the dependence on the moment $\|Y\|_{2,1} =
\int_{0}^{\infty}\sqrt{\P(|Y| > t)} \di t$ in the resulting bounds, instead of the desired moment $\E
Y^2$, as we obtain in Lemma~\ref{lem:concent} for finite classes. It is known that the
dependence on the $\|\cdot\|_{2,1}$ norm is unavoidable in some
cases \cite{ledoux1986conditions}. More importantly, we refer to the recent work
\cite{han2019convergence} discussing that the multiplier inequality can lead to
suboptimal rates (see \cite[Section 2.3.1]{han2019convergence} for more details).

\section{Statistical lower bounds and
  the necessity
  of Assumption~\ref{assumption:y-given-x}}
\label{sec:statistical-lower-bound}

The statistical guarantees obtained in the previous sections hold under no assumptions on the distribution of $X$ and under Assumption~\ref{assumption:y-given-x} on the conditional distribution of $Y$ given $X$. In this section, we show that Assumption~\ref{assumption:y-given-x} is necessary to obtain non-trivial guarantees on the excess risk without restrictions on $P_X$ and that our risk bounds are unimprovable, in a precise sense.

\begin{proposition}
  \label{prop:statistical-lower-bound}
  Fix any $n \geq 1$, $\delta \in (e^{-n}, 1)$ and any measurable function $f : \R \to \R$ satisfying $f (0) = 0$ and $\sup_{x \in \R} f (x)^2 \geq 1$.
  Then, there exists a distribution $P_X$ of $X$ such that for any estimator $\wh g$ (possibly improper and $P_X$-dependent), setting $Y = \freg (X)$ (where $\freg \in \{f, -f\})$ the following three conditions hold:
  \begin{itemize}
  \item there exists $w^* \in \R$ such that $R (g_{w^*}) = 0$;
  \item $\E [Y^2] \leq 1$;
  \item denoting $\| \freg \|_\infty = \sup_{x \in \R} |\freg (x)| = \| f \|_{\infty} \in [1, + \infty]$ we have
    \begin{equation}
      \P \bigg( R (\wh g) \geq \min \bigg( \frac{\| \freg \|_\infty^2 \cdot \log (1/\delta)}{4 n}, {1} \bigg) \bigg)
      \geq \delta.
    \end{equation}
  \end{itemize}
\end{proposition}

Before providing the proof, let us comment on the implications of this lower bound.
First, note that if the conditional second moment bound $\E [Y^2 | X ] \leq 1$ of Assumption~\ref{assumption:y-given-x} is relaxed to the weaker unconditional bound $\E Y^2 \leq 1$,
then (taking $\delta = 0.9$, and any $f$ such that $\| f \|_\infty \geq \sqrt{n}$) the worst-case excess risk of any estimator $\wh g$ is lower-bounded by an absolute constant $c$ with probability $0.9$, matching up to constants the risk of at most $1$ trivially achieved by the identically zero function.
Second, without Assumption~\ref{assumption:y-given-x} our upper bounds cannot be improved even in the ``realizable'' case where
the linear class $\Flin$
contains a perfect predictor
(that is, when $R (g) = 0$ for some $g \in \Flin$), and in particular $\var (Y|X) = 0$ almost surely.
As a result, the quantity $\sup_{x \in \R^d} \E [Y^2|X=x]$ in our assumption cannot be replaced by $\sup_{x \in \R^d} \var (Y|X=x)$.
Finally, when $Y = \freg (X)$, then the worst-case dependence on $\freg$ can be no better than $\| \freg \|_{\infty}^{2}$, as shown in the last part of the above proposition. The dependence on $m^{2}$ in our upper bounds is thus unavoidable, recalling that $m^{2} \leq \norm{\freg}_{\infty}^{2}$ whenever $Y = \freg(X)$.

We point out that the same argument as in Proposition~\ref{prop:statistical-lower-bound} shows that a dependence on $\sup_{x \in \R^d} \E [Y^2|X=x]$ is unavoidable for any conditional distribution $(P_{Y|X=x})_{x \in \R^d}$ (possibly known up to its sign), beyond the case $Y = \freg (X)$. We considered the latter special case for simplicity, and because it allows to simultaneously impose that $R (g_{w^*}) = 0$ for some $w^* \in \R^d$.
We also remark that Proposition~\ref{prop:statistical-lower-bound} is stated in dimension $d=1$ for simplicity. 
The same lower bound construction can be used for general dimension $d$ (assuming, for example, that $f$ is continuous, and imposing $|\freg| \leq |f|$), allowing one to replace the $\log (1/\delta)$ term by $d + \log (1/\delta)$.

\begin{proof}
  Let $p \in (0, 1)$ be such that $(1 - p)^n = \delta$;
  using that $1 - e^{-u} \geq (1-e^{-1}) u \geq u/2$ for $u = \log (1/\delta)/ n \in [0,1]$, we have
  \begin{equation}
    \label{eq:proof-p-bound}
    p
    = 1 - \delta^{1/n}
    \geq \frac{\log (1/\delta)}{2 n}
    \, .
  \end{equation}
  Let $x_0 \in \R \setminus \{ 0 \}$ be such that $| f(x_0) |$ is larger than $\min (\| f \|_{\infty} / \sqrt{2}, 1/\sqrt{p})$ and let $p_0 = \min (p, 1/f(x_0)^2)$. Fix the distribution of the covariates $P_X$ as follows:
  \[
      X = \begin{cases}
        0 & \text{with probability } 1 - p_0, \\
        x_0 & \text{with probability } p_0.
      \end{cases}
  \]
  Up to replacing $f$ by $-f$, assume that $f(x_0) > 0$.
  For $\eps \in \{-1, 1\}$, let $P_\eps$ denote the joint distribution of the random pair $(X, \eps f (X))$ (where the marginal distribution of $X$ is given by $P_X$ defined above), and let $R_\eps$ denote the  risk functional associated to the distribution $P_\eps$. Note that $P_\eps$ satisfies the first condition of the proposition with $w^* = \eps f(x_0) / x_0$. Also, the second condition holds since $\E Y^2 = p_0 f(x_0)^2 \leq 1$.

  We now turn to proving the third condition of this proposition.
  Let $\wh g$ be an arbitrary procedure, possibly improper and depending on $P_X$. Let $S_0 = \big( (0, 0), \dots, (0, 0) \big)$ denote a sample of $n$ points equal to $(0,0)$.  Since the quadratic loss function is convex, we may assume without loss of generality that $\wh g$ is a deterministic procedure and let $g : \R \to \R$ denote the output of $\wh g$ on the sample $S_0$, that is, $g = \widehat{g}(S_0)$. By symmetry of the problem, assume that $g(x_0) \leq 0$ and fix the distribution $P$ of $(X,Y)$ to $P_1$ (if $g(x_0) \geq 0$, we may fix $P = P_{-1}$ instead).
  Consider the event $E = \{ X_1 = \dots = X_n = 0 \}$ and note that $\P (E) = (1 - p_0)^2 \geq (1-p)^n = \delta$. Since $f(0) = 0$, on the event $E$ the observed sample is $S_0$, so that by~\eqref{eq:proof-p-bound} we have
  \begin{align}
    R (\wh g)
    &\geq \E [ (g (X) - Y)^2 \indic{X=x_0} ]= p_0 \cdot (g(x_0) - f(x_0))^2 \geq p_0 \, f(x_0)^2
    \\
    &= \min \big( p \, f(x_0)^2, 1 \big) \geq \min \Big( \frac{p \| f \|_\infty^2}{2}, 1 \Big) \geq \min \bigg( \frac{\| \freg \|_\infty^2 \cdot \log (1/\delta)}{4 n}, {1} \bigg),
  \end{align}
  which completes our proof.
\end{proof}

\section{Deferred Proofs}
\label{section:deferred-proof}
This section contains the proof of lemmas appearing in
Section~\ref{sec:optimal-estimator}. Note that rescaling the response
$Y$ by $1/m$ affects the excess risk by a multiplicative factor
$1/m^{2}$.
Thus, without loss of generality, in all the proofs of this section
we may assume that Assumption~\ref{assumption:y-given-x} holds with $m=1$.

\subsection{Proof of Lemma~\ref{lem:lonedist}}
  The proof of this lemma is based on a combination of the classical
  localization via empirical Rademacher complexities argument of \cite{bartlett2005local}
  and the covering number bounds for truncated VC-subgraph classes due to
  \cite{gyorfi2002nonparametric}.

  First, define the star-hull of
  $|\overline{\F} - \overline{\F}| = \{|f - g| : f,g \in \overline{\F}\}$
  by $\H$, and for $r \geq 0$, define its localized subsets by $\H_{r}$:
  \begin{equation}
    \label{eq:localized-star-hull}
    \H = \Bigl\{\beta|f-g| : \beta \in [0,1],\enskip f,g \in
    \overline{\F}\Bigr\},\quad
    \H_{r} = \Bigl\{h \in \H :
    \enskip \frac{1}{n}\sum_{i=1}^{n}|h(X_{i})|^{2} \leq 4r \Bigr\}.
  \end{equation}
  Let $\widehat{\psi}_{n}(r) : [0,\infty) \to \R$ denote any sub-root
  function with unique positive
  fixed-point $\widehat{r}^{*}$ (that is, a positive solution to the equation
  $\widehat{\psi}_{n}(\widehat{r}^{*}) = \widehat{r}^{*}$ (see \cite[Definition
  3.1, Lemma 3.2]{bartlett2005local}). Suppose that
  $\widehat{\psi}_{n}$ satisfies
  the following inequality for any $r \geq \widehat{r}^{*}$:
  \begin{equation}
    \label{eq:sub-root-upper-bound}
    \frac{1}{n}\E_{\varepsilon_1, \ldots, \varepsilon_n}
    \sup\limits_{h \in \H_r}\left(
      \sum_{i = 1}^n\varepsilon_i h(X_i)
    \right)
    + \frac{\log(3/\delta)}{n}
    \lesssim
    \widehat{\psi}_{n}(r),
  \end{equation}
  where $\varepsilon_{1}, \dots, \varepsilon_{n}$ is a sequence of \iid
  Rademacher random variables. Notice that for any $r \geq 0$ and any
  $h \in H_{r}$ we have $\sup_{x} |h(x)| \leq 2$ and $\E h(X)^{2} \leq 4 \E
  h(X)$. Hence, by the first part of \cite[Theorem 4.1]{bartlett2005local},
  with probability at least $1 - \delta$, the following holds simultaneously
  for all $f,g \in \overline{\F}$:
  \begin{equation}
    \label{eq:localization-thm-bbm}
    \E|f(X) - g(X)|
    \leq
    \frac{2}{n} \sum_{i=1}^{n} |f(X_{i}) - g(X_{i})|
    +c\left(
      \widehat{r}^{*} + \frac{\log(3/\delta)}{n}
    \right),
  \end{equation}
  where $c > 0$ is some universal constant.

  In the rest of the proof we show that a suitable value of $\widehat{r}^{*}$
  can be obtained by upper bounding the empirical Rademacher complexity terms
  via Dudley's entropy integral. To do so, we first need to obtain an upper
  bound on the covering numbers of the class $\H$ with
  respect to the empirical $L_{2}$ distance, defined between any $h,h' \in
  \mathcal{H}$ by $\sqrt{\frac{1}{n}\sum_{i=1}^{n} (h(X_{i}) -
  h'(X_{i}))^{2}}$. In what follows, for any class $\mathcal{G}$ and any $\gamma
  > 0$, an empirical $L_{2}$ $\gamma$-net of $\mathcal{G}$ will be denoted by
  $N(\mathcal{G}, \gamma) \subseteq \mathcal{G}$. Thus, the covering number of
  $\G$ with respect to the empirical $L_{2}$ distance at scale $\gamma$ is at most $|N(\mathcal{G}, \gamma)|$.

  Since $\H$ is a star-hull of the class $|\overline{\F} - \overline{\F}|$,
  it follows from \cite[Lemma 4.5]{mendelson2002improving} that for any $\gamma >
  0$ we have
  \begin{equation}
    \label{eq:star-hull-covering-numbers}
    |N(\H, \gamma)| \leq \left|
      N(\overline{\F} - \overline{\F}, \gamma/2)
    \right| \cdot \frac{4}{\gamma}.
  \end{equation}
  Further, noting that the Minkowski sum of $\gamma/4$ covers of
  $\overline{\F}$ forms a $\gamma/2$ cover of $\overline{\F} - \overline{\F}$
  it follows that
  \begin{equation}
    \label{eq:F-F-cover}
    |N(\overline{\F} - \overline{\F}, \gamma/2)|
    \leq
    |N(\overline{\F}, \gamma/4)|^{2}.
  \end{equation}
  Let $\overline{\F}_{+} = \{x \mapsto \max(0, f(X)) : f \in \overline{\F}\}$
  and $\overline{\F}_{-} = \{x \mapsto \min(0, f(X)) : f \in \overline{\F}\}$.
  By the same argument, it holds that
  \begin{equation}
    \label{eq:F-plus-minus-cover}
    |N(\overline{\F}, \gamma/4)|
    \leq |N(\overline{\F}_{+}, \gamma/8)| \cdot |N(\overline{\F}_{-},
    \gamma/8)|.
  \end{equation}
  Finally, plugging in the upper bounds on the covering numbers of
  $\overline{\F}_{+}$ and $\overline{\F}_{-}$ due to \cite[Theorem 9.4,
  Theorem 9.5]{gyorfi2002nonparametric}\footnote{
    See also the proof of \cite[Theorem 11.3]{gyorfi2002nonparametric} where the same bound on
    covering numbers is used.
  }, the chain of inequalities
  \eqref{eq:star-hull-covering-numbers}, \eqref{eq:F-F-cover} and
  \eqref{eq:F-plus-minus-cover} yields
  \begin{equation}
    \label{eq:H-final-covering-bound}
    \log |N(\H, \gamma)| \lesssim d \log(e/\gamma).
  \end{equation}

  Plugging in the above inequality into Dudley's
  entropy integral \cite[Theorem
  3.5.1]{gine2016mathematical} upper bound on Rademacher complexities, we obtain
  \begin{align}
  \label{eq:dudleyintegral}
    \frac{1}{n}\E_{\varepsilon_1, \ldots, \varepsilon_n}\sup\limits_{h \in \H_r}
    \left(\sum_{i = 1}^n\varepsilon_i h(X_i)\right) \nonumber
    &\lesssim \frac{1}{\sqrt{n}}
      \int_{0}^{2\sqrt{r}}\sqrt{d\log(e/\gamma)}d\gamma
    \\
    &\lesssim \sqrt{\frac{d}{n}}\sqrt{r\log (e/r)}\left(\mathbb{1}_{\{r \geq d/n \}}
      + \mathbb{1}_{\{r < d/n\}}\right)
    \\
    &\lesssim \sqrt{\frac{dr\log(n/d)}{n}}
    + \frac{d\sqrt{\log(n/d)}}{n}.
  \end{align}
  In particular, the inequality \eqref{eq:sub-root-upper-bound} is satisfied by
  the choice:
  \begin{equation}
    \label{eq:sub-root-choice}
    \widehat{\psi}_{n}(r) = c\left(
    \sqrt{\frac{dr\log(n/d)}{n}}
    + \frac{d\sqrt{\log(n/d)} + \log(3/\delta)}{n}
    \right).
  \end{equation}
  Solving the fixed-point equation $\widehat{\psi}_{n}(\widehat{r}^*) =
  \widehat{r}^*$
  yields $\widehat{r}^* \lesssim \frac{d\log(n/d) + \log(1/\delta)}{n}$.
  The claim follows by the localization theorem stated in
  \eqref{eq:localization-thm-bbm}.

\subsection{Proof of Lemma~\ref{lem:concent}}
  Fix any $f,g \in \mathcal{G}$ and recall that $\E (\ell_f - \ell_g) = R(f) - R(g)$.
  By the standard bound \cite[Theorem 2]{lugosi2019mean}, for any $\delta' \in (0,1)$, the
  choice $k(\delta') = \lceil 8\log(1/\delta') \rceil$ guarantees
  that with probability at least $1 - \delta'$ we have
  \begin{equation}
    \label{eq:median-of-means-bound-non-uniform}
    \left|R(f) - R(g) - \MOM^{k(\delta')}_{S_{n}}(\ell_{f} - \ell_{g})\right|
    \leq \sqrt{\frac{32 \mathrm{Var}(\ell_{f} - \ell_{g}) \log(1/\delta')}{n}}.
  \end{equation}
  To upper bound the variance term, first notice that
  \begin{equation}
    \ell_{f}(X, Y) - \ell_{g}(X, Y) = 2Y(g(X) - f(X)) + f(X)^2 - g(X)^2.
  \end{equation}
  Combining the above identity with the inequality $(a+b)^2 \leq 2a^2 + 2b^2$
  for any $a,b$, Assumption~\ref{assumption:y-given-x} (with $m=1$)
  and the boundedness of $f,g$, we obtain
  \begin{align}
    \Var(\ell_{f} - \ell_{g})
    &\leq 8 \E Y^2(g(X) - f(X))^2 + 2\E(f(X)^2 - g(X)^2)^2 \\
    &\leq 8 \E (g(X) - f(X))^2 + 2\E(f(X) - g(X))^{2}(f(X) + g(X))^{2} \\
    &\leq 16 \E (g(X) - f(X))^{2}.
  \end{align}
  Since the class $\G$ is finite, taking $\delta' = \delta/(2|\G|^{2})$
  the upper bound \eqref{eq:median-of-means-bound-non-uniform} extend uniformly
  to all pairs $f,g\in \G$, with probability at least $1- \delta/2$. In
  particular, for any $f,g \in \G$ it holds that
  \begin{align}
    \left|R(f) - R(g) - \MOM^{k(\delta')}_{S_{n}}(\ell_{f} - \ell_{g})\right|
    &\leq \sqrt{\frac{512 \E(g(X) - f(X))^{2} \cdot (2\log(|\G|) +
    \log(2/\delta))}{n}} \\
    &\leq \alpha \sqrt{\E(g(X) - f(X))^{2}}.
    \label{eq:lemma-conecntration-first-ineq}
  \end{align}
  This completes the proof of the first inequality.

  We will now simultaneously prove the second and the third inequalities
  appearing in the statement of this lemma.
  Note that $m=1$ ensures that for any $f,g \in \G$ we have $(f(X) - g(X))^{2} \leq
  4$ and $\E(f(X) - g(X))^{4} \leq 4\E(f(X) - g(X))^{2}$.
  Hence, for any $\delta'' \in (0,1)$ and any $f,g \in \G$, Bernstein's
  inequality ensures that with probability at least $1 - 2\delta''$ it holds
  simultaneously that
  \begin{align}
    \label{eq:population-two-times-empirical}
    &\E(f(X) - g(X))^{2}
    \leq \frac{2}{n}\sum_{i=1}^{n}(f(X)-g(X))^{2} +
    \frac{12\log(1/\delta'')}{n}, \\
    &\frac{1}{n}\sum_{i=1}^{n}(f(X_{i}) - g(X_{i}))^{2}
    \leq 2\E(f(X) - g(X))^{2} + \frac{12\log(1/\delta'')}{n}.
    \label{eq:empirical-two-times-population}
  \end{align}
  Setting $\delta'' = \delta/(4|\G|^{2})$
  the above inequalities extend uniformly to all pairs $f,g \in \G$ with
  probability at least $1-\delta/2$.
  Noting that $\frac{12\log(1/\delta'')}{n} \leq \alpha^{2}$,
  the inequality \eqref{eq:empirical-two-times-population} completes the proof
  of the third inequality of this lemma. Finally, the second inequality
  appearing in the statement of this lemma is implied (on the
  event of the first and third inequalities) by plugging in
  \eqref{eq:population-two-times-empirical} into
  \eqref{eq:lemma-conecntration-first-ineq} together with the inequality
  $\sqrt{a+b} \leq \sqrt{a} + \sqrt{b}$ valid for any $a,b \geq 0$. The proof
  of this lemma is thus complete.

\subsection{Proof of Lemma~\ref{lem:excessrisk}}
  Let $E$ denote the event of Lemma~\ref{lem:concent} (thus, $\P(E) \geq 1 -
  \delta$).
  By the definition of $g^{*}$, for any $g \in \G$ we have $R(g^{*}) - R(g)
  \leq 0$. Hence, on the event $E$ it holds simultaneously for all $g \in \G$ that
  \begin{align}
    \MOM_{S_{n}}^{k}(\ell_{g^*} - \ell_{g})
    &\le R(g^*) - R(g) + |R(g^*) - R(g) - \MOM_{S_{n}}^{k}(\ell_{g^*} - \ell_{g})|
    \\
    &\le  \sqrt{2}\alpha\sqrt{\frac{1}{n}\sum\limits_{i = 1}^n(g^{*}(X_i) - g(X_i))^2}
    + \alpha^2.
  \end{align}
  In particular, on the event $E$ the function $g^{*} \in \wh \G$,
  which completes the first part of the proof.

  We now turn to proving the second part of this lemma. Since $g^{*} \in \wh
  \G$, by the definition of $\wh \G$, for any $g \in \wh \G$ we have
  \begin{equation}
    \MOM_{S_{n}}^{k}(\ell_g - \ell_{g^*})
    \leq
    \sqrt{2}\alpha\sqrt{\frac{1}{n}\sum\limits_{i = 1}^n(g(X_i) -
    g^*(X_i))^2} + \alpha^2.
  \end{equation}
  Hence, on the event $E$, by the third inequality of
  Lemma~\ref{lem:concent}, for any $g \in \wh \G$ it holds that
  \begin{align*}
    R(g) - R(g^*)
    &\le \left|R(g) - R(g^*) - \MOM_{S_{n}}^{k}(\ell_g - \ell_{g^*})\right| +
    \MOM_{S_{n}}^{k}(\ell_g - \ell_{g^*})
    \\
    &\leq 2\sqrt{2}\alpha\sqrt{\frac{1}{n}\sum\limits_{i = 1}^n(g(X_i) - g^*(X_i))^2}
    + 2\alpha^2
    \\
    &\leq 4\alpha\sqrt{\E (g(X) - g^*(X))^{2}} + 5\alpha^2.
  \end{align*}
  By the definition of the $L_{2}$ diameter of the class $\wh \G$ and by the
  fact that $g^*,g \in \wh \G$, it follows
  that $\sqrt{\E (g(X) - g^*(X))^{2}} \leq \mathcal{D}(\wh \G)$ and hence our
  proof is complete.

\subsection{Proof of Lemma~\ref{lemma:minmax-mom-estimator}}
  First observe that
   \begin{align}
     R(\wh{g})
     &=
     R(g^{*}) + \left(R(\wh{g}) - R(g^*)
     - \MOM_{S_{n}}^{k}(\ell_{\widehat{g}} - \ell_{g^{*}})\right)
     + \MOM_{S_{n}}^{k}(\ell_{\widehat{g}} - \ell_{g^{*}})
     \\
     &\leq
     R(g^{*}) + \sup_{g \in \G}\left|R(g) - R(g^*)
     - \MOM_{S_{n}}^{k}(\ell_{g} - \ell_{g^{*}})\right|
     + \MOM_{S_{n}}^{k}(\ell_{\widehat{g}} - \ell_{g^{*}})
     \\
     &\leq
     R(g^{*}) + \alpha\mathcal{D}(\G)
     + \MOM_{S_{n}}^{k}(\ell_{\widehat{g}} - \ell_{g^{*}}),
     \label{eq:min-max-mom-estimator-intermediate-bound}
   \end{align}
  where the last line follows via an application of Lemma~\ref{lem:concent}.
  Further, notice that by the definition of $\wh g$ we have
  \[
    \MOM_{S_{n}}^{k}(\ell_{\wh{g}} - \ell_{g^*})
    \le \max\limits_{g \in \wh \G}\MOM_{S_{n}}^{k}(\ell_{\wh{g}} - \ell_{g})
    \le \max\limits_{g \in \wh \G}\MOM_{S_{n}}^{k}(\ell_{g^*} - \ell_{g}).
  \]
  At the same time, on the event of Lemma~\ref{lem:concent},
  for all $g \in \G$ we have
  \[
    \MOM_{S_{n}}^{k}(\ell_{g^*} - \ell_{g})
    \le R(g^*) - R(g) + \alpha\mathcal{D}(\G)
    \leq \alpha\mathcal{D}(\G).
  \]
  Combining the above inequality with
  \eqref{eq:min-max-mom-estimator-intermediate-bound} concludes our proof.
  
  \paragraph{Acknowledgements.} We would like to thank Manfred Warmuth for several useful discussions. 
  T.V.~is supported by the EPSRC and MRC through the OxWaSP CDT programme (EP/L016710/1). N.Z.~is funded in part by ETH Foundations of Data Science (ETH-FDS).

{\small	
\bibliography{references}
}
\bibliographystyle{abbrv}

\end{document}